\documentclass[3p,10pt,nopreprintline]{elsarticle}

\usepackage{amsmath}
\usepackage{amsthm}
\usepackage{amssymb}
\usepackage{graphicx}
\usepackage{epsfig}
\usepackage[normal]{subfigure}
\usepackage{float}
\usepackage{paralist}
\usepackage{caption2}
\usepackage{hyperref}
\usepackage{aliascnt}
\usepackage[T1]{fontenc}
\usepackage{natbib}
\usepackage{CJK}
\usepackage{color}
\usepackage[table]{xcolor}
\usepackage{colortbl,booktabs}
\hypersetup{
colorlinks=true,
linkcolor=black
}

\newtheoremstyle{theorem}{6pt}{6pt}{\itshape}{}{\bfseries}{.}{.5em}{}
\newtheoremstyle{definition}{6pt}{6pt}{\upshape}{}{\bfseries}{.}{.5em}{}

\theoremstyle{theorem}
\newtheorem{theorem}{Theorem}[section]

\newaliascnt{corollary}{theorem}

\aliascntresetthe{corollary}

\newaliascnt{lemma}{theorem}
\newtheorem{lemma}[lemma]{Lemma}
\aliascntresetthe{lemma}

\newaliascnt{sublemma}{theorem}

\aliascntresetthe{sublemma}

\theoremstyle{definition}

\newtheorem{definition}{Definition}[section]

\newtheorem{remark}{Remark}[section]

\newaliascnt{proposition}{theorem}

\aliascntresetthe{proposition}

\newcommand{\dif}{{\mathrm d}}

\newcommand{\bn}{\begin{eqnarray}}
\newcommand{\en}{\end{eqnarray}}
\newcommand{\bnn}{\begin{eqnarray*}}
\newcommand{\enn}{\end{eqnarray*}}
\renewcommand{\div}{ {\rm div }  }

\newcommand{\N}{\mathbb{N}}

\newcommand{\al}{\alpha}

\newcommand{\frD}{\mathfrak{D}}

\newcommand{\D}{\nabla}

\allowdisplaybreaks

\numberwithin{equation}{section}

\begin{document}

\begin{frontmatter}

\title{Local existence of classical solutions to the 3D isentropic compressible Navier-Stokes-Poisson equations with degenerate viscosities and vacuum}


\author[label1]{Peng Lu}
\address[label1]{School of Mathematical Sciences, Shanghai Jiao Tong University, Shanghai 200240, P.R.China;}
\ead{lp95@sjtu.edu.cn}

\author[label1]{Shaojun Yu\corref{cor2}}
\cortext[cor2]{Corresponding author. }
\ead{edwardsmith123@sjtu.edu.cn}

\begin{abstract}
We consider the isentropic compressible Navier–Stokes–Poisson equations
with degenerate viscousities and vacuum in  a three-dimensional torus. The local
well-posedness of classical solution is established by introducing a ``quasi-symmetric hyperbolic''--``degenerate elliptic'' coupled structure to control the behavior of the velocity of the fluid near the vacuum and  give some uniform estimates. In particular, the initial data allows vacuum in an open set and we do not need any initial compatibility conditions.

 \end{abstract}

\begin{keyword}
Density-dependent viscosities, compressible Navier–Stokes-Poisson equations, vacuum,  local classical solutions.
\end{keyword}

\end{frontmatter}



\section{Introduction}
	
In this paper, we consider the compressible Navier–Stokes–Poisson (NSP) equations for the dynamics of
charged particles of electrons (see \cite{Markowich}). The equations can be written as
\begin{equation}\label{CNSP}
\begin{cases}
\rho_t+\div(\rho u)=0,\\
(\rho u)_t+\div(\rho u\otimes u)+\D P(\rho)=\text{div} \mathbb{S}+\rho\D\Phi,\\
\Delta\Phi=\rho-b,
\end{cases}
\end{equation}
for $t>0$ and $x\in\mathbb{T}^3$, with initial data
\begin{align}\label{data}
t=0:(\rho, u, \Phi)=(\rho_0(x), u_0(x), \Phi_0(x)) \quad \text{for} \quad x\in\mathbb{T}^3,
\end{align}
where $\mathbb{T}^3$ is a three-dimensional torus, $x=(x_1,x_2,x_3)^\top\in \mathbb{T}^3 $ and $t\geq 0$ denote the spatial coordinate and time coordinate, respectively. $\rho\geq 0$ is  the mass density,  and $u = (u^1, u^2, u^3)^\top$ is  the fluid velocity.  The pressure  $P$ of the polytropic fluid satisfies
\begin{equation}\label{Pecv}
P(\rho)=A\rho^\gamma,
\end{equation}
where $A>0$ and $\gamma>1$ are the gas constants. $\Phi$ is  the electrostatic potential satisfying
\begin{equation}
    m_\Phi(t)\triangleq \int_{\mathbb{T}^3}\Phi(t,x) \mathrm{d}x=0, \quad \forall ~ t\geq 0.
\end{equation}

Moreover, $\mathbb{S}$ is the viscosity stress tensor given by
\begin{equation}\label{tensor:T}
\mathbb{S}=2\mu(\rho)\frD(u)+\lambda(\rho)\div u\mathbb{I}_3,
\end{equation}
where 
$$\frD(u) =\frac{\D u+(\D u)^\top}{2},$$
is the deformation tensor, $\mathbb{I}_3$ is the $3 \times 3$ identity matrix, and 
\begin{equation}\label{density-dependent}
\mu(\rho)=\al\rho^\delta, \quad \lambda(\rho)=\beta\rho^\delta,
\end{equation}
where $\mu$ is the shear viscosity coefficient, and $\lambda+\frac{2}{3}\mu$ is the bulk viscosity coefficient, $(\al, \beta,\delta)$  are constants satisfying
\begin{equation}\label{constants}
\al>0, \quad 2\al+3\beta\geq 0, \quad \delta>1,
\end{equation}
the constant $b>0$ is the doping profile, which describes the density of fixed, positively charged background ions.

From the mathematical point of view, the NSP equations are the compressible Navier-Stokes
equations coupled with the Poisson equation. Indeed, if there is absent of the electrostatic effects, then system (1.1)–(1.3) will be reduced to the barotropic compressible Navier–Stokes equations. For the constant viscous fluid, there is a lot of literature on the well-posedness of classical solutions to  isentropic compressible Navier–Stokes equations. When $\inf_{x} \rho_0(x)>0$, the local well-posedness of  classical solutions follows from the standard symmetric hyperbolic-parabolic structure satisfying the well-known Kawashima's condition (cf. \cite{Kawa1983Sys}), which has been extended to a global one by Matsumura-Nishida \cite{Matsu1980The} near the nonvacuum equilibrium. When $\inf_{x} \rho_0(x)=0$, the first main issue is the degeneracy of the time evolution operator, which makes it difficult to describe the behavior of the velocity field near the vacuum. For this case, the local-in-time well-posedness of strong solutions with vacuum was firstly solved by Cho-Choe-Kim \cite{Cho2004Unique} and Cho-Kim \cite{Cho2006On} in $\mathbb{R}^3$, where they introduced an initial compatibility condition to compensate the lack of a positive lower bound of density.
Later, Huang-Li-Xin \cite{Huang2012Glo} extended the local existence to a global one under some initial smallness assumption in $\mathbb{R}^3$. Jiu-Li-Ye \cite{Jiu2014Glo} proved the global existence of classical solution with arbitrarily large data and vacuum in $\mathbb{R}$.

When viscosity coefficients are density-dependent, the Navier-Stokes system have received extensive attentions in recent years, especially for the case with vacuum, where the well-posedness of solutions become more challenging due to the degenerate viscosity. In fact, the high order regularity estimates of the velocity in \cite{Cho2004Unique,Huang2012Glo,Jiu2014Glo} ($\delta=0$) strongly rely on the uniform ellipticity of the Lam\'e operator.
While for $\delta>0$, $\mu(\rho)$ vanishes as the density function connects to vacuum continuously, thus it is difficult to adapt the approach of the constant viscosity case. A remarkable discovery of a new mathematical entropy function was made by Bresch-Desjardins \cite{Bre2003Exi}
for the viscosity satisfying some mathematical relation, which provides additional regularity on some derivative of the density. This observation was applied widely in proving the global existence of weak solutions with vacuum for Navier-Stokes equations and some related models; see Bresch-Desjardins \cite{Bre2003Exi}, Bresch-Vasseur-Yu \cite{Bresch2022Glo}, Jiu-Xin \cite{Jiu2008The}, Mellet-Vasseur \cite{Mellet2007On}, Vasseur-Yu \cite{Vasseur2016Exi}, and so on. Then, we turn our attention to the study of classical solutions. 
When  $\delta=1$, Li-Pan-Zhu \cite{Li2017On} obtained the local existence of 2-D classical solution with far field vacuum, which also applies to the 2-D shallow water equations. 
When $1 < \delta \leq \min \left\{3, \frac{\gamma+1}{2}\right\}$, by making full use of the symmetrical structure of the hyperbolic operator and the weak smoothing effect of the elliptic operator, Li-Pan-Zhu \cite{Li2019On} established the local existence of 3-D classical solutions with arbitrarily large data and vacuum, see also Geng-Li-Zhu \cite{Geng2019Vanishing} for more related results, and Xin-Zhu \cite{Xin2021Glo} for the global existence of classical solution under some initial smallness assumptions in homogeneous Sobolev space. When $0<\delta<1$, Xin-Zhu \cite{Xin2021Well} obtained the local existence of 3-D local classical solution with far field vacuum, Cao-Li-Zhu \cite{Cao2022Glo} proved the global existence of 1-D classical solution with large initial data.   Some other interesting results and discussions can also be found in  Germain-Lefloch \cite{Germain2016Finite}, Guo-Li-Xin \cite{Guo2012Lagrange}, Lions \cite{Lions1998Math}, Yang-Zhao \cite{Yang2002Vacuum}, and the references therein.

Concerning the NSP system, there are also extensive studies about the local and  global well-posedness of solutions.
For the constant viscous fluid (i.e., $\delta=0$ in \eqref{density-dependent}),  Donatelli \cite{Don2003Loc} obtained the local and global existence of weak solutions to 3D isentropic Navier-Stokes-Poisson equations with vacuum in a bounded domain.
Tan-Zhang \cite{Tan2010Strong} obtained the local existence of strong solution to 3D isentropic Navier-Stokes-Poisson equations with vacuum in a bounded domain. 
Li \cite{Li2010Opi} established the global existence and large time behavior to classical solution for Cauchy problem with small data. 
Zheng \cite{Zhe2012NA} established the global existence for Cauchy problem in Besov space. Tan-Wang-Wang \cite{Tan2015sta} established the global existence of classical solutions and obtained  the time decay rates of the solution. Liu-Xu-Zhang \cite{Liu2020Glo} established the global well-posedness of strong solutions with large oscillations and vacuum. 

For the case of density-dependent viscosity coefficients (i.e., $\delta> 0$ in \eqref{density-dependent}), the problem is much more challenging due to the degeneration near the vacuum and hence the obtained results are limited. Ducomet, etc. \cite{Ducomet2010On} studied the global stability of the weak solutions to the Cauchy problem of the NSP equations with non-monotone pressure as $\gamma>\frac{4}{3}$. In \cite{Ducomet2011sph} they also considered
Cauchy problem for the NSP equations of spherically symmetric motions, including both constant viscosities and density-dependent viscosities, and proved the global stability of the weak solutions provided that $\gamma>1$. Zlotinik \cite{Zlotnik2008On} studied the long time behavior of the spherically symmetric weak solutions near a hard core by giving global-in-time bounds for the solutions. Ye-Dou \cite{Ye2017Glo} studied the global existence of weak solutions to the compressible NSP equations with density-dependent viscosities for $\mu(\rho)=\rho, \lambda(\rho)=0$ in \eqref{density-dependent} in a three dimensional torus. 

It should be pointed out that, in spite of the above significant achievements, a lot of questions remain open, including the local well-posedness of classical solutions in multi-dimensions with vacuum. In this paper, we consider the 3D isentropic compressible Navier-Stokes-Poisson equations with degenerate viscosities and vacuum in a torus, where the viscosities depend on the density in a super-linear power law (i.e., $\delta>1$ in \eqref{density-dependent}), and obtain the local existence of classical solutions.

Here and throughout this paper, we adopt the following simplified notations: for any $p, r\in[1,\infty]$ and integer $k, s\geq 0$, we denote
\begin{equation*}
\begin{aligned}
&|f|_p  =\|f\|_{L^p\left(\mathbb{T}^3\right)},\quad \|f\|_s=\|f\|_{H^s\left(\mathbb{T}^3\right)}, \quad D^{k, r} =\left\{f \in L_{l o c}^1\left(\mathbb{T}^3\right):\left|\nabla^k f\right|_r<+\infty\right\}, \\
& D^k =D^{k, 2},\quad |f|_{D^{k, r}}=\|f\|_{D^{k, r}\left(\mathbb{T}^3\right)},\quad |f|_{D^k}=\|f\|_{D^k\left(\mathbb{T}^3\right)},
 \quad \int f=\int_{\mathbb{T}^3} f \mathrm{~d} x,\\
&\|(f, g)\|_X=\|f\|_X+\|g\|_X, \quad \|f\|_{X\cap Y}=\|f\|_X+\|f\|_Y.
\end{aligned}
\end{equation*}
	

\begin{definition}\label{220}
  Let $T>0$ be a finite constant. A solution $(\rho,u, \Phi)$ to the Cauchy problem \eqref{CNSP}--\eqref{data} is called a regular solution in $[0,T]\times \mathbb{T}^3$ if $(\rho,u,\Phi)$ satisfies this problem in the sense of distribution  and:
  \begin{itemize}
\item[(1)]  $\rho \geq 0, \quad \rho^{\frac{\delta-1}{2}} \in C\left([0, T] ; H^3\right), \quad \rho^{\frac{\gamma-1}{2}} \in C\left([0, T] ; H^3\right);$
\item[(2)] $ u \in C([0, T] ; H^{s^{\prime}}) \cap L^{\infty}\left(0, T; H^3\right),  \quad \rho^{\frac{\delta-1}{2}}  \nabla^4 u \in L^2\left(0, T; L^2\right); $
\item[(3)] $ u_t+u \cdot \nabla u=0$ ~ {\rm as} ~ $\rho(t, x)=0;$
\item[(4)] $\Phi\in C\left([0, T] ; H^3\right)\cap L^{\infty}\left(0, T; H^4\right),$ 
\end{itemize}
where $s^{\prime} \in[2,3)$ is a constant.
\end{definition}
 
Our main theorem can be stated as follows.
\begin{theorem}\label{main theory}
Assume that 
$$\delta\in (1,2]\cup \{3\}, ~ \gamma >1.$$
If the initial data $(\rho_0,u_0,\Phi_0)$ satisfies the following regularity conditions:
\begin{equation}\label{203}
\rho_0\geq 0, ~ \left(\rho_0^{\frac{\gamma-1}{2}},~ \rho_0^{\frac{\delta-1}{2}},  ~ u_0\right)\in H^3,
\end{equation}
then there exists a positive time $T_*$ and a unique regular solution $(\rho, u,\Phi)(t, x)$ in $[0, T_*]\times\mathbb{T}^3$ to the problem \eqref{CNSP}--\eqref{data}  satisfying: 
\begin{equation*}
\begin{aligned}
\sup _{0 \leq t \leq T_*} & \left(\|\rho^{\frac{\gamma-1}{2}}\|_3^2+\|\rho^{\frac{\delta-1}{2}}\|_3^2+\|u\|_3^2+\|\Phi\|_4\right)+ \int_0^t |\rho^{\frac{\delta-1}{2}}  \nabla^4 u|_2^2 \mathrm{d} s \leq C^0,
\end{aligned}
\end{equation*}
for arbitrary constant $s^{\prime} \in[2,3)$ and positive constant $C^0=C^0(A, \gamma, \delta, \alpha,\beta,\rho_0, u_0)$. Actually, $(\rho, u,\Phi)$ satisfies the problem \eqref{CNSP}--\eqref{data} classically in positive time $\left(0, T_*\right]$.
\end{theorem}

\begin{remark}
The initial data allows vacuum in an open set and we do not need any initial compatibility conditions.   And there are also no smallness conditions imposed on initial data. 
\end{remark}
\begin{remark}
 The restriction $\delta\leq 3$ is only used to demonstrate the regular solution we obtained in Theorem \ref{main theory} is a classical one. For details, see \eqref{301} in \S \ref{section 3.6}.
\end{remark}

The rest of this paper is organized as follows. \S \ref{Preliminaries} is
dedicated for the preliminary lemmas to be used later.  \S \ref{proof of well-posedness} is devoted to proving the well-posedness of local classical solutions, i.e., Theorem \ref{main theory}.
 
\section{Preliminaries}\label{Preliminaries}

The following well-known Gagliardo–Nirenberg inequality will be used (see \cite{Ladyzenskaja}).
\begin{lemma}\label{GN}
For $p\in[2,6]$, $q\in(1,\infty)$ and $r\in(3,\infty)$, there exists a constant $C>0$ which may depend on $q,r$  such that for $f\in H^1$ and $g\in L^q \cap D^{1,r}$, it holds
\begin{equation}
|f|_p\leq C|f|_2^{\frac{6-p}{2p}}|\D f|_2^{\frac{3p-6}{2p}},
\end{equation}
\begin{equation}
|g|_\infty\leq C|g|_q^{\frac{q(r-3)}{3r+q(r-3)}}|\D g|_r^{\frac{3r}{3r+q(r-3)}}.
\end{equation}
\end{lemma}

The next several lemmas contain some Sobolev inequalities on the product estimates, the interpolation estimates, the composite function estimates, etc., which
can be found in many works, say Majda \cite{Majda}.

\begin{lemma}[\cite{Majda}]\label{Lem:2.2}
For constants $s\in\N_+$, $r, p, q\in[1,\infty]$ satisfying
$$\frac{1}{r}=\frac{1}{p}+\frac{1}{q},$$
and functions $f,g\in W^{s,p}\cap W^{s,q}$, there exists a constant $C$ only depending on $s$, such that
\begin{equation}
\left|\D^s(fg)-f\D^sg\right|_r\leq C\left(|\D f|_p|\D^{s-1}g|_q+|\D^sf|_p|g|_q\right),
\end{equation}
\begin{equation}
\left|\D^s(fg)-f\D^sg\right|_r\leq C\left(|\D f|_p|\D^{s-1}g|_q+|\D^sf|_q|g|_p\right),
\end{equation}
where $\D^sf$ $(s>1)$ stands for the set of all partial derivatives $\partial_x^\xi f$ with $|\xi|=s$.
\end{lemma}

\begin{lemma}[\cite{Majda}]\label{Lem:2.3}
If functions $f, g \in H^s$ and $s > \frac{3}{2}$, then $fg\in H^s$, and there
exists a constant $C$ only depending on $s$ such that
\begin{equation}
\|fg\|_s\leq C\|f\|_s\|g\|_s.
\end{equation}
\end{lemma}

\begin{lemma}[\cite{Majda}]\label{Lem:2.4}
If $u \in H^s$, then for any $r\in[0,s]$, there exists a constant $C$ only
depending on $s$ such that
\begin{equation}
\|u\|_r\leq C|u|_2^{1-\frac{r}{s}}\|u\|_s^{\frac{r}{s}}.
\end{equation}
\end{lemma}

\begin{lemma}[\cite{Majda}]\label{Lem:2.5} 
$(1)$ If $f, g \in H^s \cap L^\infty$ and $|\zeta | \leq s$, then there exists a constant $C$ only depending on $s$ such that
\begin{equation}
\left\|\partial_x^\zeta(fg)\right\|_s\leq C\left(|f|_\infty|\D^sg|_2+|g|_\infty|\D^sf|_2\right).
\end{equation}
$(2)$ Let $u(x)$ be a continuous function taking its values in some open set $G$ such
that $u \in H^s \cap L^\infty$, and $g(u)$ be a smooth vector-valued function on $G$. Then
for any $s\geq 1$, there exists a constant $C$ only depending on $s$ such that
\begin{equation}
|\D^sg(u)|_2\leq C\left\|\frac{\partial g}{\partial u}\right\|_{s-1}|u|_\infty^{s-1}|\D^su|_2.
\end{equation}
\end{lemma}

\begin{lemma}
 [\cite{Majda}]\label{104} If function sequence $\left\{w_n\right\}_{n=1}^{\infty}$ converges weakly in a Hilbert space $X$ to $w$, then $w_n$ converges strongly to $w$ in $X$ if and only if
\begin{equation*}
\|w\|_X \geq \limsup_{n \rightarrow \infty}\left\|w_n\right\|_X.
\end{equation*}
\end{lemma}

The next one will show some compactness results from the Aubin-Lions Lemma.
\begin{lemma}[\cite{Simon}]\label{22}
Let $X_0$, $X$ and $X_1$ be three Banach spaces with $X_0 \subset X \subset X_1$. Suppose that $X_0$ is compactly embedded in $X$ and that $X$ is continuously embedded in $X_1$. Then:
\begin{itemize}
    \item Let $G$ be bounded in $L^p(0,T; X_0)$ for $1 \leq p < \infty$, and $\frac{\partial G}{\partial t}$ be bounded in $L^1(0, T; X_1)$, then $G$ is relatively compact in $L^p(0,T; X)$.
    \item Let $F$ be bounded in $L^\infty(0,T; X_0)$ and $\frac{\partial F}{\partial t}$ be bounded in $L^p(0,T; X_1)$ with $p >1$,then $F$ is relatively compact in $C(0,T; X)$.
\end{itemize}
\end{lemma}

The last lemma will be used to show the time continuity for the higher order terms of our solution:
\begin{lemma}[\cite{Bol2003Semi}]\label{221}
 If $f(t, x) \in L^2\left(0, T; L^2\right)$, then there exists a sequence $s_k$ such that
$$
s_k \rightarrow 0, \quad \text { and } \quad s_k\left|f\left(s_k, x\right)\right|_2^2 \rightarrow 0, \quad \text { as } \quad k \rightarrow+\infty.
$$
\end{lemma}


\section{Local-in-time well-posedness of solutions }\label{proof of well-posedness}
\subsection{Reformulation}
By introducing three new quantities,
\begin{equation*}
\phi= \rho^{\frac{\gamma-1}{2}},\quad    \varphi=\rho^{\frac{\delta-1}{2}},  
\end{equation*}
equations \eqref{CNSP}--\eqref{data} can be formally rewritten as
\begin{equation*}
\begin{cases}
&\varphi_t+u\cdot \D  \varphi+\frac{\delta-1}{2} \varphi  \div u=0, \\
&\phi_t+u \cdot \D   \phi+\frac{\gamma-1}{2} \phi  \div  u=0, \\
&u_t+u \cdot \D  u+\frac{A \gamma}{\gamma-1}\D  \phi^2+\varphi^2 L(u)=\frac{ \delta}{\delta-1}Q(u)\cdot \D  \varphi^2+\D  \Phi, \\
&\Delta\Phi=\varphi^{\frac{2}{\delta-1}}-b,\\
&(\varphi, W,\Phi)|_{t=0}=\left(\varphi_0(x), W_0(x),\Phi_0(x)\right), \quad x \in \mathbb{T}^3, 
\end{cases}
\end{equation*}
which can be rewritten into a new system that consists of a transport equation for $\varphi$, a linear elliptic equation for $\Phi$ and  a ``quasi-symmetric hyperbolic''--``degenerate elliptic'' coupled system with some special lower order source terms for $(\phi, u)$:
\begin{equation}\label{4}
\begin{cases}
&\varphi_t+u\cdot \D  \varphi+\frac{\delta-1}{2} \varphi  \div u=0,\\
&\underbrace{A_0 W_t+ \sum_{j=1}^{3}A_j(W)\partial_j W}_{\text {symmetric hyperbolic }}+\underbrace{\varphi^2  \mathbb{L}(W)}_{\text {degenerate elliptic }}= \underbrace{\mathbb{Q}(W)H(\varphi)+\left(\begin{matrix}
0\\
a_1\D  \Phi
\end{matrix}\right)}_{\text {lower order source }}   , \\
&\Delta\Phi=\varphi^{\frac{2}{\delta-1}}-b,\\
&(\varphi, W,\Phi)|_{t=0}=\left(\varphi_0(x), W_0(x),\Phi_0(x)\right), \quad x \in \mathbb{T}^3, 
\end{cases}
\end{equation}
where $W(x)=(\phi,u)^{\top}(x), W_0(x)=(\phi_0,u_0)^{\top}(x)$ and
\begin{align}
&A_0=
\left( \begin{array}{cc}
1 & 0 \\
0 & a_1 \mathbb{I}_3
 \end{array} \right),
 ~~~ A_j(W)=
\left( \begin{array}{cc}
u_j & \frac{\gamma-1}{2} \phi e_j^\top\\
\frac{\gamma-1}{2} \phi  e_j & a_1 u_j \mathbb{I}_3
 \end{array} \right), ~ j=1,2,3, \nonumber\\ 
 &\mathbb{L}(W)=
 \left( \begin{array}{c}
0 \\
a_1 L(u)
 \end{array} \right), ~~~ L(u)=-\alpha(\Delta u+\D  \div  u)-\beta \D  \div u, \nonumber\\
 & H(\varphi)=
 \left( \begin{array}{c}
0 \\
\D  \varphi^2
 \end{array} \right), ~~~\mathbb{Q}(W)=
\left( \begin{array}{cc}
0 &0 \\
0 & \frac{a_1 \delta}{\delta-1} Q(u)
 \end{array} \right),\nonumber\\ 
&Q(u)=\alpha(\D  u+(\D  u)^\top)+\beta\div u \mathbb{I}_3,~~~a_1=\frac{(\gamma-1)^2}{4A\gamma}.\nonumber
\end{align}

\begin{theorem}  \label{29}
 If initial data $\left(\varphi_0, W_0,\Phi_0 \right)$ satisfies
\begin{equation}
\rho_0 \geq 0, \quad (\varphi_0, W_0)\in H^3,
\end{equation}
then there exists a time $T_*>0$ and a unique regular solution $(\varphi,\phi,u,\Phi )$ in $\left[0, T_*\right] \times \mathbb{T}^3$ to the problem \eqref{4} satisfying: 
\begin{equation}\label{53}
\begin{aligned}
& \varphi \in C\left(\left[0, T^*\right] ; H^3\right), \quad  \phi \in C\left(\left[0, T^*\right] ; H^3\right), \\
&\Phi\in C\left([0, T^*] ; H^3\right)\cap L^{\infty}\left(0, T^*; H^4\right),\\
& u \in C\left(\left[0, T^*\right] ; H^{s^{\prime}}\right) \cap L^{\infty}\left(0, T^* ; H^3\right), \quad s^{\prime} \in[2,3), \\
& \varphi \nabla^4 u \in L^2\left(0, T^* ; L^2\right), \quad  u_t \in C\left(\left[0, T^*\right] ; H^1\right) \cap L^2\left(0, T^* ; D^2\right).
\end{aligned}
\end{equation}

Moreover, we have 
\begin{equation*}
\begin{aligned}
\sup _{0 \leq t \leq T_*} & \left(\|\phi\|_3^2+\|\varphi\|_3^2+ \|u\|_3^2\right)(t)  + \int_0^t |\varphi  \nabla^4 u|_2^2 \mathrm{d} s \leq C^0,
\end{aligned}
\end{equation*}
for arbitrary constant $s^{\prime} \in[2,3)$ and positive constant $C^0=C^0( \gamma, \delta,\varphi_0,W_0,\Phi_0)$.
\end{theorem}

\subsection{Linearization}
\begin{equation}\label{linear}
\begin{cases}
&\varphi_t+v\cdot \D  \varphi+\frac{\delta-1}{2} \tilde{\varphi}  \div v=0,\\
&A_0 W_t+ \sum\limits_{j=1}^{3}A_j(V)\partial_j W+(\varphi^2+\eta^2)  \mathbb{L}(W)= \mathbb{Q}(V)H(\varphi)+\left(\begin{matrix}
0\\
a_1\D  \Phi
\end{matrix}\right)   , \\
&\Delta\Phi=\varphi^{\frac{2}{\delta-1}}-b,\\
&(\varphi, W,\Phi)|_{t=0}=\left(\varphi_0(x), W_0(x),\Phi_0(x)\right), \quad x \in \mathbb{T}^3, 
\end{cases}
\end{equation}
where $\eta \in(0,1]$ is a constant, $W=(\phi, u)^{\top}, V=(\tilde{\phi}, v)^{\top}$ and $W_0=\left(\phi_0, u_0\right)^{\top}$. $(\tilde{\varphi},\tilde{\phi})$ are all known functions and $v=(v_1,v_2,v_3)^\top\in \mathbb{R}^3$ is a known vector satisfying the initial assumption $(\tilde{\varphi},\tilde{\phi},v)(t=0, x)=\left(\varphi_0,  \phi_0, u_0\right)(x)$ and
\begin{equation}\label{12}
\begin{split}
& \tilde{\varphi}\in C([0,T];H^3),\quad \tilde{\varphi}_t\in C([0,T];H^2), \quad \tilde{\phi}\in C([0,T];H^3),\\
& \tilde{\phi}_t\in C([0,T];H^2), \quad  v\in C([0,T];H^{s^{\prime}})\cap L^{\infty}(0,T;H^3),   \\
&\tilde{\varphi}\D ^4 v \in L^2\left(0, T; L^2\right), \quad v_t \in C\left([0, T] ; L^2\right) \cap L^2\left(0, T; D^1\right).
\end{split}
\end{equation}
Moreover, we assume that 
\begin{equation}\label{18}
\varphi_0\geq 0, \quad \phi_0\geq 0, \quad (\varphi_0, W_0)\in H^3.
\end{equation}

Now we have the following existence result of a strong solution $(\varphi,\phi,u,\Phi)$ to \eqref{linear} by the standard methods at least when $\eta>0$:
\begin{lemma}\label{21}
Assume that the initial data $(\varphi_0, W_0)$ satisfy \eqref{18}. Then there exists a constant $T^*>0$, such that for any $T\in[0,T^*]$, there exists a unique strong solution $(\varphi,\phi,u,\Phi)$ in $[0,T]\times\mathbb{T}^3$ to \eqref{linear} when $\eta>0$ such that
\begin{align}
& \varphi\in C\left([0,T]; H^3\right), \quad \phi\in C\left([0,T]; H^3\right), \quad \Phi\in L^\infty\left(0,T; H^4\right)\cap C\left([0,T]; H^3\right), \\
&u\in C\left([0,T]; H^3\right)\cap L^2\left(0,T; D^4\right), \quad u_t\in C\left([0,T]; H^1\right)\cap L^2\left(0,T; D^2\right).\nonumber
\end{align}
\end{lemma}
\begin{proof}
First, the local-in-time existence and regularities of a unique solution $\varphi$ in $(0, T)\times\mathbb{T}^3$ to the equations $\eqref{linear}_1$ can be obtained by the standard theory of Characteristics, see \cite[Section 3.2]{Evans}.

Next, since $\varphi\in C([0,T];H^3)$ and $\frac{2}{\delta-1}>2$, we have $\varphi^{\frac{2}{\delta-1}}\in C([0,T];H^2)$. Consequently, the existence and $L^\infty(0,T;H^4)$ regularity of a unique solution $\Phi$ in $(0, T)\times\mathbb{T}^3$ to the equations $\eqref{linear}_4$ can be obtained by the standard theory of elliptical equations, see \cite[Section 6.2--6.3]{Evans} for details. To obtain the $C([0,T];H^3)$ regularity of $\Phi$, we note that $\eqref{linear}_1$ implies
\begin{align*}
\left(\varphi^{\frac{2}{\delta-1}}\right)_t=\varphi^{\frac{3-\delta}{\delta-1}}\left(-v\cdot\D\varphi-\frac{\delta-1}{2}\tilde{\varphi}\div v\right)\in L^\infty\left(0,T;H^1\right).
\end{align*}

Therefore, for any $t,s\in[0,T]$, we have
\begin{align*}
\left\|\varphi^{\frac{2}{\delta-1}}(t)-\varphi^{\frac{2}{\delta-1}}(s)\right\|_1\leq C\int_s^t\left\|\left(\varphi^{\frac{2}{\delta-1}}\right)_t(\tau)\right\|_1\dif\tau\to 0,
\end{align*}
as $s\to t$. Then standard elliptical estimates imply
\begin{align*}
\|\Phi(t)-\Phi(s)\|_3\leq C\left\|\varphi^{\frac{2}{\delta-1}}(t)-\varphi^{\frac{2}{\delta-1}}(s)\right\|_1\to 0,
\end{align*}
as $s\to t$. This gives the $C([0,T];H^3)$ regularity of $\Phi$.

Finally, when $\eta>0$, based on the regularities of $\varphi$ and $\Phi$, it is not difficult to solve $W$ from the linear symmetric hyperbolic-parabolic coupled system $\eqref{linear}_3$ to complete the proof, see \cite[Chapter 7]{Evans}.
\end{proof}

\subsection{A priori estimates}\label{estimates}
Assume that
\begin{equation}\label{81}
\left\|\varphi_0\right\|_3+\left\|\phi_0\right\|_3+\left\|u_0\right\|_3 \leq c_0,
\end{equation}
and
\begin{equation}\label{26}
\begin{aligned}
& \sup _{0 \leq t \leq T^*}\left(\|\tilde{\varphi}(t)\|_1^2+\|\tilde{\phi}(t)\|_1^2+\|v(t)\|_1^2\right)+\int_0^{T^*} \left|\tilde{\varphi} \D ^2 v\right|_2^2 \dif  t \leq c_1^2, \\
& \sup _{0 \leq t \leq T^*}\left(\|\tilde{\varphi}(t)\|_2^2+\|\tilde{\phi}(t)\|_2^2+\|v(t)\|_2^2\right) +\int_0^{T^*} \left|\tilde{\varphi} \D ^3  v\right|_2^2 \dif  t\leq c_2^2, \\
& \sup _{0 \leq t \leq T^*}\left(\|\tilde{\varphi}(t)\|_3^2+\|\tilde{\phi}(t)\|_3^2+\|v(t)\|_3^2\right)+\int_0^{T^*} \left|\tilde{\varphi} \D ^4 v\right|_2^2 \dif  t \leq c_3^2. \\
\end{aligned}
\end{equation}

\begin{lemma}\label{pr1}
Let $T_1=\min \left\{T^*,\left(1+Cc_3\right)^{-2}\right\}$.  Then for $0 \leq t \leq T_1$,
\begin{equation*}
\begin{aligned}
\|\varphi(t)\|_3\leq Cc_0, \quad \|\varphi_t(t)\|_2\leq C c_3^2,\quad \|\Phi(t)\|_4\leq Cc_0.
\end{aligned}
\end{equation*}
\end{lemma}

\begin{proof}
Apply the operator $\partial_x^\zeta(0\leq |\zeta|\leq 3)$ to $\eqref{linear}_1,$ we have 
\begin{equation}\label{6}
  (\partial_x^\zeta \varphi)_t+v \cdot \D  \partial_x^\zeta \varphi=-(\partial_x^\zeta(v\cdot \D  \varphi)-v\cdot\D  \partial_x^\zeta \varphi)- \partial_x^\zeta(\tilde{\varphi}\div v).
\end{equation}
Multiplying both sides of $\eqref{6}$ by $\partial_x^\zeta \varphi$, integrating over $\mathbb{T}^3$, we get
\begin{equation}
\begin{split}
   \frac{1}{2}\frac{\text{d}}{\text{d}t}|\partial_x^\zeta \varphi|_2^2\leq~& C|\div  v|_\infty |\partial_x^\zeta \varphi|_2^2+C|\partial_x^\zeta(v\cdot \D  \varphi-v\cdot\D  \partial_x^\zeta \varphi)|_2|\partial_x^\zeta \varphi|_2 + C|\partial_x^\zeta(\tilde{\varphi}\div v)|_2|\partial_x^\zeta \varphi|_2.
\end{split}
\end{equation}
First we consider the case where  $|\zeta|\leq 2.$
\begin{equation}
\begin{aligned}
 & |\partial_x^\zeta(v\cdot \D  \varphi-v\cdot\D  \partial_x^\zeta \varphi)|_2
  \leq C(|\D  v\cdot \D  \varphi|_2+|\D  v\cdot \D ^2 \varphi|_2+|\D ^2 v\cdot \D  \varphi|_2)\leq C\|v\|_3\|\varphi\|_2,\\
 & |\partial_x^\zeta(\tilde{\varphi}\div v)|_2\leq C\|\tilde{\varphi}\|_2\|v\|_3,
\end{aligned}
\end{equation}
which yields that 
\begin{equation}
  \frac{\text{d}}{\text{d}t} \|\varphi\|_2\leq C\|v\|_3\|\varphi\|_2+C\|\tilde{\varphi}\|_2\|v\|_3.
\end{equation}
Then, according to Gronwall's inequality, one has 
\begin{equation}
   \|\varphi\|_2\leq (\|\varphi_0\|_2+Cc_3^2 t)\exp(Cc_3 t)\leq Cc_0,
\end{equation}
for $0\leq T_1=\min\{T^*,(1+Cc_3)^{-2}\}.$

When $|\zeta|= 3,$
\begin{equation*}
\begin{aligned}
|\partial_x^\zeta(v\cdot \D  \varphi)-v\cdot\D  \partial_x^\zeta \varphi|_2
\leq~ & C(|\D ^3 v\cdot \D  \varphi|_2+|\D ^2 v\cdot \D ^2 \varphi|_2+|\D  v\cdot \D ^3 \varphi|_2)\\
\leq~ & C\|v\|_3\|\varphi\|_3,
\end{aligned}
\end{equation*}
\begin{equation*}
\begin{aligned}
|\partial_x^\zeta(\tilde{\varphi}\div v)|_2&\leq C(|\tilde{\varphi}\D ^3 \div  v|_2+|\D  \tilde{\varphi}\cdot \D ^2 \div  v|_2+|\D ^2 \tilde{\varphi}\D  \div  v|_2 +|\D ^3 \tilde{\varphi}\div  v|_2)\\
&\leq C(|\tilde{\varphi}\D ^4 v|_2+\|\tilde{\varphi}\|_3\|v\|_3).
\end{aligned}
\end{equation*}

Then we have 
\begin{equation}
\frac{\dif }{\dif t} \|\varphi\|_3\leq C(\|v\|_3\|\varphi\|_3+\|v\|_3\|\tilde{\varphi}\|_3+|\tilde{\varphi}\D ^4 v|_2).
\end{equation}

According to Gronwall's Inequality, we obtain that
\begin{equation}\label{7}
\|\varphi\|_3\leq \left(\|\varphi_0\|_3+c_3^2 t+\int_0^t|\tilde{\varphi}\D ^4 v|_2 \dif s\right) \exp(Cc_3t).
\end{equation}

Noting that 
\begin{equation}\label{8}
\int_0^t |\tilde{\varphi}\D ^4 v|_2\dif s \leq t^{\frac{1}{2}}\left(\int_0^t |\tilde{\varphi}\D ^4 v|_2^2 \dif s \right)^{\frac{1}{2}}\leq c_3 t^{\frac{1}{2}},
\end{equation}
it follows from \eqref{7}--\eqref{8} that 
\begin{equation}\label{rho:3}
\|\varphi\|_3\leq Cc_0,
\end{equation}
for $0\leq t\leq T_1.$ From 
\begin{equation}
\varphi_t=-v\cdot \D  \varphi- \tilde{\varphi}  \div v ,
\end{equation}
we easily have 
\begin{equation*}
\begin{aligned}
|\varphi_t|_2\leq~ & C(|v|_\infty|\D \varphi|_2+|\tilde{\varphi}|_\infty|\div  v|_2)\leq Cc_2^2,\\
|\varphi_t|_{D^1}\leq~ & C(|v|_\infty|\D ^2\varphi|_2+|\D  v|_\infty|\D  \varphi|_2+|\D  \tilde{\varphi}|_{\infty}|\D  v|_2+|\tilde{\varphi}|_\infty|\D ^2 v|_2)\leq Cc_3^2,\\
|\varphi_t|_{D^2}\leq~ & C(|v|_\infty|\D ^3\varphi|_2+|\D  v|_\infty|\D ^2 \varphi|_2+|\D ^2 v|_2|\D  \varphi|_\infty\\
&+|\D  \tilde{\varphi}|_{\infty}|\D ^2 v|_2+|\tilde{\varphi}|_\infty|\D ^3 v|_2)\leq Cc_3^2.
\end{aligned}
\end{equation*}

By \eqref{rho:3} and standard elliptical estimates, we also have
\begin{equation*}
\|\Phi(t)\|_4\leq C\|\varphi-b\|_2\leq Cc_0.
\end{equation*}

Thus, we complete the proof of the lemma.
\end{proof}

\begin{lemma}
Let $(\varphi, W,\Phi)$ be the solutions to \eqref{linear}. Then 
\begin{equation*}
\begin{split}
\|W(t)\|_1^2+\int_0^t |\sqrt{\varphi^2+\eta^2}\D ^2 u|_2^2 \dif s  \leq C c_0^2,
\end{split}
\end{equation*}
for $0 \leq t \leq T_2\triangleq \min \left\{T_1,\left(1+Cc_3\right)^{-3}\right\}$.
\end{lemma}
\begin{proof}
Applying $\partial_x^\zeta$ to $\eqref{linear}_3$, we have 
\begin{equation}\label{10}
\begin{aligned}
&A_0 \partial_x^\zeta W_t+\sum_{j=1}^{3} A_j(V)\partial_j \partial_x^\zeta W+ (\varphi^2+\eta^2) \mathbb{L}(\partial_x^\zeta W)\\
=~& \partial_x^\zeta \mathbb{Q}(V)H(\varphi)-\Big(\partial_x^\zeta\big(\sum_{j=1}^{3}A_j(V)\partial_j W\big)-\sum_{j=1}^{3} A_j(V)\partial_j \partial_x^\zeta W\Big)\\
& -\Big( \partial_x^\zeta\big( (\varphi^2+\eta^2) \mathbb{L}(W) \big)- (\varphi^2+\eta^2) \mathbb{L}(\partial_x^\zeta W) \Big)\\
 &+\Big( \partial_x^\zeta \big(\mathbb{Q}(V)H(\varphi)\big)- \partial_x^\zeta \mathbb{Q}(V)H(\varphi)\Big)+\left(\begin{matrix}
0\\
a_1\D  \partial_x^\zeta \Phi
\end{matrix}\right).
\end{aligned} 
\end{equation}
Multiplying \eqref{10} by $\partial_x^\zeta W$ on both sides and integrating over $\mathbb{T}^3,$
\begin{align}
   &\frac{1}{2} \frac{\dif }{\dif t}\int (\partial_x^\zeta W)^\top A_0 \partial_x^\zeta W+a_1 \alpha|\sqrt{\varphi^2+\eta^2}\D  \partial_x^\zeta u|_2^2 +a_1 (\alpha+\beta) |\sqrt{\varphi^2+\eta^2}\div  \partial_x^\zeta u|_2^2  \nonumber \\
  =~& \frac{1}{2} \int (\partial_x^\zeta W)^\top \div A(V) \partial_x^\zeta W -a_1\alpha\int(\D (\varphi^2+\eta^2)\cdot(\D  \partial_x^\zeta u+\div  \partial_x^\zeta u \mathbb{I}_3) )\cdot \partial_x^\zeta u \nonumber\\
  &+\frac{ a_1\delta}{\delta-1}\int (\D  \varphi^2 \partial_x^\zeta Q(v))\cdot \partial_x^\zeta u \label{13}\\
  &-\int \Big(\partial_x^\zeta\big(\sum_{j=1}^{3}A_j(V)\partial_j W\big)-\sum_{j=1}^{3} A_j(V)\partial_j \partial_x^\zeta W\Big) \cdot \partial_x^\zeta W\nonumber\\
  & -a_1 \int\Big( \partial_x^\zeta\big( (\varphi^2+\eta^2) L(u) \big)- (\varphi^2+\eta^2) L(\partial_x^\zeta u) \Big)\cdot\partial_x^\zeta u\nonumber\\
  &+\frac{ a_1\delta}{\delta-1}\int \Big( \partial_x^\zeta \big(\D  \varphi^2 \cdot  Q(v)\big)- \D  \varphi^2\cdot  Q(\partial_x^\zeta v)\Big)\cdot \partial_x^\zeta u \nonumber\\
  &+ a_1\int \D  \partial_x^\zeta \Phi \cdot \partial_x^\zeta u \nonumber\\
  \triangleq &\sum_{i=1}^{7}\text{I}_i, \nonumber
\end{align}
where $A(V)=(A_1(V),A_2(V),A_3(V))$ and $\div  A(V)=\sum\limits_{j=1}^{3} \partial_j A_j(V)$.

When $|\zeta|\leq 1,$ we have
\begin{align}
\mathrm{I}_1&=\frac{1}{2} \int (\partial_x^\zeta W)^\top \div A(V) \partial_x^\zeta W  \leq C|\D  V|_\infty|\partial_x^\zeta W|_2^2 \nonumber\\
  &\leq C\|V\|_3 |\partial_x^\zeta W|_2^2 \leq C c_3 |\partial_x^\zeta W|_2^2, \nonumber\\
\mathrm{I}_2&=-a_1\alpha\int\big(\D (\varphi^2+\eta^2)\cdot(\D  \partial_x^\zeta u+\div    \partial_x^\zeta  u\mathbb{I}_3) \big)\cdot \partial_x^\zeta u  \nonumber\\
   & \leq C|\D  \varphi|_\infty |\varphi \D  \partial_x^\zeta u|_2|\partial_x^\zeta u|_2  \nonumber\\
   &\leq \frac{a_1 \alpha}{10} |\sqrt{\varphi^2+\eta^2}\D  \partial_x^\zeta u|_2^2+C c_0^2|\partial_x^\zeta u|_2^2, \nonumber\\
\mathrm{I}_3&=\frac{ a_1\delta}{\delta-1}\int (\D  \varphi^2 \partial_x^\zeta Q(v))\cdot \partial_x^\zeta u \nonumber\\
  &\leq C|\varphi|_\infty|\D  \varphi|_\infty|\D  \partial_x^\zeta v|_2|\partial_x^\zeta u|_2 \nonumber\\
  &\leq Cc_3^3|\partial_x^2 u|_2\leq  Cc_3^3|\partial_x^2 u|_2^2+Cc_3^3, \nonumber\\
\mathrm{I}_4&=-\int \Big(\partial_x^\zeta\big(\sum_{j=1}^{3}A_j(V)\partial_j W\big)-\sum_{j=1}^{3} A_j(V)\partial_j \partial_x^\zeta W\Big) \cdot \partial_x^\zeta W \nonumber \\
  &\leq C|\D  V|_\infty |\D  W|_2^2\leq C c_3|\D  W|_2^2, \nonumber \\
\mathrm{I}_5&=-a_1\alpha\int \Big( \partial_x^\zeta\big( (\varphi^2+\eta^2) L(u) \big)- (\varphi^2+\eta^2) L(\partial_x^\zeta u) \Big)\cdot\partial_x^\zeta u \nonumber \\
  &\leq C|\D  \varphi|_\infty|\varphi L(u)|_2|\partial_x^\zeta u|_2\leq \frac{a_1 \alpha}{10} |\sqrt{\varphi^2+\eta^2}\D ^2 u|_2^2+C c_0^2|\partial_x^\zeta u|_2^2, \nonumber \\
\mathrm{I}_{6}&=\frac{ a_1\delta}{\delta-1}\int \Big( \partial_x^\zeta \big(\D  \varphi^2 \cdot  Q(v)\big)- \D  \varphi^2\cdot  Q(\partial_x^\zeta v)\Big)\cdot \partial_x^\zeta u \nonumber \\
&\leq C(|\varphi|_\infty|\D  v|_\infty|\D ^2 \varphi|_2+|\D  v|_\infty|\D  \varphi|_3|\D  \varphi|_6)|\partial_x^\zeta u|_2  \nonumber \\
&\leq Cc_3^3 |\partial_x^\zeta u|_2 \leq Cc_3^3+Cc_3^3|\partial_x^\zeta u|_2^2, \nonumber \\
\mathrm{I}_{7}&=a_1\int \D  \partial_x^\zeta \Phi \cdot \partial_x^\zeta u\leq C \|\Phi\|_2|\partial_x^\zeta u|_2\nonumber \\
&\leq C c_0+Cc_0 |\partial_x^\zeta u|_2^2.
\end{align}

Then, it yields that
\begin{equation}
\frac{1}{2} \frac{\dif }{\dif t}\int (\partial_x^\zeta W)^\top A_0 \partial_x^\zeta W+ \frac{1}{2} a_1\alpha |\sqrt{\varphi^2+\eta^2}\D ^2 u|_2^2\leq Cc_3^3\|W\|_1^2+Cc_3^3.
\end{equation}

By Gronwall's inequality, we have 
\begin{equation}
\begin{split}
  \|W\|_1^2+ \int_0^t |\sqrt{\varphi^2+\eta^2}\D ^2 u|_2^2\dif s \leq & ~ C(\|W_0\|_1^2+c_3^3 t) \exp(Cc_3^3 t)\leq C c_0^2,
  \end{split}
\end{equation}
for $0\leq t\leq T_2=\min\{T_1,(1+Cc_3)^{-3}\}$.
\end{proof}

\begin{lemma}
Let $(\varphi,W, \Phi)$ be the solutions to \eqref{linear}.  Then 
\begin{equation*}
\begin{aligned}
&|W(t)|_{D^2}^2+\int_0^t |\sqrt{\varphi^2+\eta^2}\D ^3 u|_2^2 \dif s  \leq C c_0^2,\\
&\|\phi_t(t)\|_1+|u_t(t)|_2+\int_0^t|u_s|_{D^1}^2 \dif s\leq Cc_3^3,
\end{aligned}
\end{equation*}
for $0 \leq t \leq T_2$.
\end{lemma}

\begin{proof}
Now we consider the terms on the righthand side of \eqref{13} when $|\zeta|=2$. It follows from  Lemma \ref{GN}, \eqref{13},  H\"older's inequality,  Young's inequality and the integration by parts that 
\begin{align}
\mathrm{I}_1&=\frac{1}{2} \int (\partial_x^\zeta W)^\top \div A(V) \partial_x^\zeta W  \leq C|\D  V|_\infty|\partial_x^\zeta W|_2^2 \nonumber\\
  &\leq C\|V\|_3 |\partial_x^\zeta W|_2^2 \leq C c_3 |\partial_x^\zeta W|_2^2, \nonumber\\
\mathrm{I}_2&=-a_1\alpha\int\big(\D (\varphi^2+\eta^2)\cdot(\D  \partial_x^\zeta u+\div    \partial_x^\zeta  u\mathbb{I}_3) \big)\cdot \partial_x^\zeta u  \\
   & \leq C|\D  \varphi|_\infty |\varphi \D  \partial_x^\zeta u|_2|\partial_x^\zeta u|_2  \leq \frac{a_1 \alpha}{10} |\sqrt{\varphi^2+\eta^2}\D  \partial_x^\zeta u|_2^2+C c_0^2|\partial_x^\zeta u|_2^2, \nonumber\\
\mathrm{I}_3&=\frac{ a_1\delta}{\delta-1}\int (\D  \varphi^2 \partial_x^\zeta Q(v))\cdot \partial_x^\zeta u \nonumber\\
  &\leq C|\varphi|_\infty|\D  \varphi|_3|\D  \partial_x^\zeta v|_6|\partial_x^\zeta u|_2 \nonumber\\
  &\leq Cc_3^3|\partial_x^\zeta u|_2\leq  Cc_3^3|\partial_x^\zeta u|_2^2+Cc_3^3, \nonumber\\
\mathrm{I}_4&=-\int \Big(\partial_x^\zeta\big(\sum_{j=1}^{3}A_j(V)\partial_j W\big)-\sum_{j=1}^{3} A_j(V)\partial_j \partial_x^\zeta W\Big) \cdot \partial_x^\zeta W \nonumber \\
  &\leq C|\D  V|_\infty |\D  W|_2^2\leq C c_3|\D  W|_2^2, \nonumber \\
\mathrm{I}_5&=-a_1\int \Big( \partial_x^\zeta\big( (\varphi^2+\eta^2) L(u) \big)- (\varphi^2+\eta^2) L(\partial_x^\zeta u) \Big)\cdot\partial_x^\zeta u \nonumber \\
  &\leq C(|\D  \varphi|_\infty^2|\D ^2 u|_2^2+|\D  \varphi|_\infty|\varphi \D ^3 u|_2|\D ^2 u|_2) \nonumber\\ 
  &\leq \frac{a_1 \alpha}{10} |\sqrt{\varphi^2+\eta^2}\D ^3 u|_2^2+C c_0^2|\D ^2 u|_2^2, \nonumber \\
\mathrm{I}_{6}&=\frac{ a_1\delta}{\delta-1}\int \Big( \partial_x^\zeta \big(\D  \varphi^2 \cdot  Q(v)\big)- \D  \varphi^2\cdot  Q(\partial_x^\zeta v)\Big)\cdot \partial_x^\zeta u \nonumber \\
&\leq C(|\varphi|_\infty|\D ^2 v|_6|\D ^2 \varphi|_3+|\D \varphi|_\infty^2 |\D ^2 v|_2+|\D  \varphi|_\infty|\D ^2 \varphi|_3|\D  v|_6 \nonumber\\
&\quad +|\varphi|_\infty|\D ^3 \varphi|_2|\D  v|_\infty+|\D  \varphi|_\infty|\D  v|_\infty|\D ^2 \varphi|_2)|\D ^2 u|_2  \nonumber \\
&\leq Cc_3^3 |\D ^2 u|_2 \leq Cc_3^3+Cc_3^3|\D ^2 u|_2^2, \nonumber \\
\mathrm{I}_{7}&=a_1\int \D  \partial_x^\zeta \Phi \cdot \partial_x^\zeta u\leq C \|\Phi\|_3|\partial_x^\zeta u|_2\nonumber \\
&\leq C c_0+Cc_0 |\partial_x^\zeta u|_2^2.
\end{align}

Then, it yields that
\begin{equation}
\frac{1}{2} \frac{\dif }{\dif t}\int (\partial_x^\zeta W)^\top A_0 \partial_x^\zeta W+ \frac{1}{2} a_1\alpha |\sqrt{\varphi^2+\eta^2}\D ^3 u|_2^2\leq Cc_3^3|W|_{D^2}^2+Cc_3^3.
\end{equation}

According to Gronwall's inequality, we have 
\begin{equation}
|W|_{D^2}^2+ \int_0^t |\sqrt{\varphi^2+\eta^2}\D ^3 u|_2^2\dif s\leq C(|W_0|_{D^2}^2+c_3^3 t) \exp(Cc_3^3 t)\leq C c_0^2,
\end{equation}
for $0\leq t\leq T_2$.

From the relation that
\begin{equation}\label{14}
\phi_t=-v\cdot \D  \phi-\frac{\gamma-1}{2}\tilde{\phi}\div  u,
\end{equation}
we have
\begin{equation}
\begin{split}
|\phi_t|_2&\leq C|v\cdot \D  \phi+\tilde{\phi}\div  u|_2\\
&\leq C(|v|_\infty|\D  \phi|_2+|\tilde{\phi}|_\infty|\div  u|_2)\\
&\leq Cc_2^2,\\
|\phi_t|_{D^1}&\leq C(|\D  v|_\infty|\D  \phi|_2+|v|_\infty|\D ^2 \phi|_2+|\D  \tilde{\phi}|_\infty|\D  u|_2+|\tilde{\phi}|_\infty|\D ^2 u|_2)\\
&\leq Cc_3^2.
\end{split}
\end{equation}

From the relation that
\begin{equation}\label{15}
\begin{split}
&u_t+v\cdot \D  u+\frac{2A\gamma}{\gamma-1}\tilde{\phi}\D  \phi-(\varphi^2+\eta^2)(\alpha\Delta u+(\alpha+\beta)\D  \div u) \\
=~&\frac{\delta}{\delta-1}(\alpha(\D  v+(\D  v)^\top)+\beta \div v\mathbb{I}_3)\cdot \D \varphi^2+\D  \Phi,
\end{split}
\end{equation}
we have
\begin{equation}
\begin{split}
|u_t|_2\leq~& C (|v|_\infty|\D  u|_2+|\tilde{\phi}|_\infty|\D  \phi|_2+|\varphi^2+\eta^2|_\infty |\D ^2 u |_2+|\varphi|_\infty|\D  \varphi|_\infty|\D  v|_2+|\D  \Phi|_2)\\
 \leq~ & Cc_2^3.
\end{split}
\end{equation}

For $|u_t|_{D^1}$,
\begin{equation}\label{16}
\begin{split}
|u_t|_{D^1}\leq~ &C(|\D  v|_6 |\D  u|_3+|v|_\infty|\D ^2 u|_2+|\D  \tilde{\phi}|_3|\D  \phi|_6+|\tilde{\phi}|_\infty|\D ^2 \phi|_2\\
&+|\sqrt{\varphi^2+\eta^2}|_\infty|\sqrt{\varphi^2+\eta^2}\D ^3 u|_2+|\varphi|_\infty|\D  \varphi|_\infty|\D ^2 u|_2\\
&+|\D  \varphi|_\infty^2|\D  v|_2+|\varphi|_\infty|\D ^2 \varphi|_3|\D  v|_6+|\varphi|_\infty|\D  \varphi|_\infty|\D ^2 v|_2+|\D ^2 \Phi|_2)\\
\leq~ & Cc_2^3+Cc_0|\sqrt{\varphi^2+\eta^2}\D ^3 u|_2,
\end{split}
\end{equation}
which implies that
\begin{equation}
\int_{0}^{t} |u_s|_{D^1}^2 \dif s\leq C\int_{0}^{t}(c_2^6+c_0^2|\sqrt{\varphi^2+\eta^2}\D ^3 u|_2^2)\dif s\leq Cc_2^3,
\end{equation}
for $0\leq t\leq T_2.$
\end{proof}

\begin{lemma}\label{17}
Let $(\varphi,W,\Phi)$ be the solutions to \eqref{linear}.  Then 
\begin{equation*}
\begin{aligned}
&|W(t)|_{D^3}^2+\int_0^t |\sqrt{\varphi^2+\eta^2}\nabla^4 u|_2^2 \mathrm{d}s   \leq C c_0^2,\\
&|\phi_t(t)|_{D^2}^2+|u_t(t)|_{D^1}^2+\int_0^t|u_s|_{D^2}^2 \text{d}s\leq Cc_3^6,
\end{aligned}
\end{equation*}
for $0 \leq t \leq T_3\triangleq \min\{T_2, (1+Cc_3)^{-4}\}$.
\end{lemma}
\begin{proof}
Now we consider the terms on the righthand side of \eqref{13} when $|\zeta|=3$. It follows from  Lemma \ref{GN}, \eqref{13},  H\"older's inequality and Young's inequality that 
\begin{align}
\mathrm{I}_1&=\frac{1}{2} \int (\partial_x^\zeta W)^\top \text{div}A(V) \partial_x^\zeta W  \leq C|\nabla V|_\infty|\partial_x^\zeta W|_2^2 \nonumber\\
  &\leq C\|V\|_3 |\partial_x^\zeta W|_2^2 \leq C c_3 |\partial_x^\zeta W|_2^2, \nonumber\\
\mathrm{I}_2&=-a_1\alpha\int\big(\nabla(\varphi^2+\eta^2)\cdot(\nabla \partial_x^\zeta u+\text{div}   \partial_x^\zeta  u\mathbb{I}_3) \big)\cdot \partial_x^\zeta u  \\
   & \leq C|\nabla \varphi|_\infty |\varphi \nabla^4 u|_2|\partial_x^\zeta u|_2 \leq \frac{a_1 \alpha}{10} |\sqrt{\varphi^2+\eta^2}\nabla^4 u|_2^2+C c_0^2|\partial_x^\zeta u|_2^2, \nonumber\\
\mathrm{I}_3&=\frac{ a_1\delta}{\delta-1}\int (\nabla \varphi^2 \partial_x^\zeta Q(v))\cdot \partial_x^\zeta u \nonumber\\
  &\leq C(|\nabla \varphi|_\infty^2|\nabla^3 v|_2|\nabla^3 u|_2+|\nabla^2 \varphi|_3|\varphi \nabla^3 u|_6|\nabla ^3 v|_2+|\nabla \varphi|_\infty |\varphi \nabla^4 u|_2|\nabla^3 v|_2) \nonumber\\
  &\leq Cc_3^3|\nabla^3 u|_2^2+Cc_3^4+\frac{a_1 \alpha}{10} |\sqrt{\varphi^2+\eta^2}\nabla^4 u|_2^2, \nonumber\\
\mathrm{I}_4&=-\int \Big(\partial_x^\zeta\big(\sum_{j=1}^{3}A_j(V)\partial_j W\big)-\sum_{j=1}^{3} A_j(V)\partial_j \partial_x^\zeta W\Big) \cdot \partial_x^\zeta W \nonumber \\
  &\leq C(|\nabla V|_\infty|\nabla^3 W|_2^2+|\nabla^2 V|_3|\nabla^2 W|_6|\nabla^3 W|_2+|\nabla ^3 V|_2|\nabla W |_\infty|\nabla W|_2) \nonumber \\
  &\leq Cc_3|\nabla^3 W|_2^2, \nonumber\\
\mathrm{I}_5&=-a_1\int \Big( \partial_x^\zeta\big( (\varphi^2+\eta^2) L(u) \big)- (\varphi^2+\eta^2) L(\partial_x^\zeta u) \Big)\cdot\partial_x^\zeta u \nonumber \\
  &\leq C(|\nabla \varphi|_\infty|\nabla^2 \varphi|_3|\nabla^2 u|_6|\nabla^3 u|_2+|\nabla^3\varphi|_2|\nabla^2 u|_3|\varphi\nabla^3 u|_6+|\nabla \varphi|_\infty^2|\nabla^3 u|_2^2 \nonumber\\ 
  & \quad + |\nabla^2\varphi|_3|\varphi \nabla^3 u|_6|\nabla^3 u|_2+|\nabla\varphi|_\infty|\varphi \nabla^4 u|_2|\nabla^3 u|_2)       \nonumber\\ 
  &\leq \frac{a_1 \alpha}{10} |\sqrt{\varphi^2+\eta^2}\nabla^4 u|_2^2+C c_0^2|\nabla^3 u|_2^2, \nonumber \\
\mathrm{I}_{6}&=\frac{ a_1\delta}{\delta-1}\int \Big( \partial_x^\zeta \big(\nabla \varphi^2 \cdot  Q(v)\big)- \nabla \varphi^2\cdot  Q(\partial_x^\zeta v)\Big)\cdot \partial_x^\zeta u \nonumber \\
&\leq C(|\nabla v|_\infty|\nabla^2\varphi|_3|\nabla^2\varphi|_6|\nabla^3 u|_2+|\nabla v|_\infty|\nabla\varphi|_\infty|\nabla^3\varphi|_2|\nabla^3 u|_2+|\nabla^3 \varphi|_2|\nabla^2 v|_6|\varphi\nabla^3 u|_3 \nonumber\\
&\quad +|\nabla^3 \varphi|_2|\nabla v|_\infty|\varphi\nabla^4 u|_2+|\nabla \varphi|_\infty|\nabla^2\varphi|_3|\nabla^2 v|_6|\nabla^3 u|_2+|\varphi|_3|\nabla^2v|_6|\nabla^3 \varphi|_2|\nabla^3 u|_2   \nonumber \\
&\quad +|\nabla \varphi|_\infty^2|\nabla^3 v|_2|\nabla^3 u|_2+|\nabla^3 v|_2|\nabla^2 \varphi|_3|\varphi\nabla^3 u|_6) \nonumber \\
&\leq Cc_3^3 |\nabla^3 u|_2^2 +Cc_3^4+\frac{a_1\alpha}{10}|\sqrt{\varphi^2+\eta^2}\nabla^4 u|_2^2, \nonumber \\
\mathrm{I}_{7}&=a_1\int \D  \partial_x^\zeta \Phi \cdot \partial_x^\zeta u\leq C \|\Phi\|_4|\nabla^3 u|_2\nonumber \\
&\leq C c_0+Cc_0 |\nabla^3 u|_2^2
\end{align}
where $\mathrm{I}_6$ was processed via integration by parts.
Then, it yields that
\begin{equation}
\frac{1}{2} \frac{\mathrm{d}}{\mathrm{d}t}\int (\partial_x^\zeta W)^\top A_0 \partial_x^\zeta W+ \frac{1}{2} a_1\alpha |\sqrt{\varphi^2+\eta^2}\nabla^4 u|_2^2\leq Cc_3^3|W|_{D^3}^2+Cc_3^4.
\end{equation}

According to Gronwall's inequality, we have 
\begin{equation}
|W|_{D^3}^2+ \int_0^t |\sqrt{\varphi^2+\eta^2}\nabla^4 u|_2^2\mathrm{d}s\leq C(|W_0|_{D^2}^2+c_3^4 t) \exp(Cc_3^3 t)\leq C c_0^2,
\end{equation}
for $0\leq t\leq T_3=\min\{T_1,(1+Cc_3)^{-4}\}$.

For $|\phi_t|_{D^2}$, from \eqref{14}, we have
\begin{equation}
\begin{split}
|\phi_t|_{D^2}\leq~ & C(|\nabla^2 v|_3|\nabla \phi|_6+|\nabla v|_3|\nabla^2\phi|_6+|v|_\infty|\nabla^3 \phi|_2\\
&+|\nabla^2\tilde{\phi}|_3|\nabla u|_6+|\nabla\tilde{\phi}|_3|\nabla^2 u|_6+|\tilde{\phi}|_\infty|\nabla^3 u|_2)\\
\leq~ &Cc_3^2.
\end{split}
\end{equation}

It follows from \eqref{16} that
\begin{equation}
\begin{split}
|u_t|_{D^1}\leq~ & Cc_2^3+Cc_0|\sqrt{\varphi^2+\eta^2}\nabla^3 u|_2\leq  Cc_2^3.
\end{split}
\end{equation}

For $|u_t|_{D^2}$, from \eqref{15}, we have
\begin{equation}
\begin{split}
|u_t|_{D^2}\leq~ & C(\|v\|_3\|u\|_3+\|\tilde{\phi}\|_3\|\phi\|_3+\|\varphi\|_3\|u\|_3+\|\varphi\|_3^2\|u\|_3
+|\varphi|_\infty|\sqrt{\varphi^2+\eta^2}\nabla^4 u|_2)\\
\leq~ & Cc_3^3+Cc_3|\sqrt{\varphi^2+\eta^2}\nabla^4 u|_2,
\end{split}
\end{equation}
which implies that 
\begin{equation}
\begin{split}
\int_{0}^{t} |u_t|_{D^2}^2 \mathrm{d}s\leq &C\int_{0}^{t} \Big(c_3^6+c_3^2(|\sqrt{\varphi^2+\eta^2}\nabla^4 u|_2^2)\Big)\dif s
\leq  Cc_3^4,
\end{split}
\end{equation}
for $0\leq t\leq T_3.$
\end{proof}

Combining the estimates obtained in Lemmas \ref{pr1}--\ref{17}, we have 
\begin{equation}
\begin{split}
&\|\varphi\|_3^2\leq Cc_0^2,\quad \|\varphi_t\|_2\leq Cc_3^2,\quad \|\Phi(t)\|_4\leq Cc_0,\\
& \|W\|_1^2+\int_{0}^{t}|\varphi\nabla^2 u|_2^2\mathrm{d}s  \leq Cc_0^2,\\
& |W|_{D^2}^2+\int_{0}^{t}|\varphi\nabla^3 u|_2^2\mathrm{d}s  \leq Cc_0^2,\\
&\|\phi_t\|_1^2+|u_t|_2^2+\int_{0}^{t}|u_t|_{D^1}^2 \mathrm{d}s \leq Cc_3^6,\\
&|W|_{D^3}^2+\int_{0}^{t}|\varphi\nabla^4 u|_2^2\mathrm{d}s  \leq Cc_0^2,\\
&\|\phi_t\|_{D^2}^2+|u_t|_{D^1}^2+\int_{0}^{t}|u_t|_{D^2}^2 \mathrm{d}s \leq Cc_3^6,
\end{split}
\end{equation}
for $0\leq t\leq T_3=\min\{T_2,(1+Cc_3)^{-4}\}.$

Therefore, if we define the constants $c_i(i=1,2,3)$ and $T^*$ by 
\begin{equation}
c_1=c_2=c_3=C^{\frac{1}{2}}c_0, \quad T^*=\min\{T,(1+Cc_0)^{-4}\},
\end{equation}
then we deduce that 
\begin{equation}\label{20}
\begin{split}
&\sup_{0 \leq t\leq T^*} \left(\|\varphi(t)\|_1^2  +\|\phi(t)\|_1^2+\|u(t)\|_1^2\right)+\int_{0}^{T^*}|\varphi \nabla^2 u|_2^2 \mathrm{d}t \leq c_1^2,\\
&\sup_{0 \leq t\leq T^*}  \left(|\varphi(t)|_{D^2}^2 +|\phi(t)|_{D^2}^2+\|u(t)\|_{D^2}^2\right)+\int_{0}^{T^*}|\varphi \nabla^3 u|_2^2 \mathrm{d}t\leq c_2^2,\\
&\sup_{0 \leq t\leq T^*} \left( |\varphi(t)|_{D^3}^2  +|\phi(t)|_{D^3}^2+|u(t)|_{D^3}^2+\|\Phi\|_4\right)+\int_{0}^{T^*}|\varphi \nabla^4 u|_2^2 \mathrm{d}t\leq c_3^2,\\
&\sup_{0 \leq t\leq T^*} \left( \|u_t\|_1^2+\|\phi_t\|_2^2+\|\varphi_t\|_2^2 \right)+\int_{0}^{T^*}|u_t|_{D^2}^2\mathrm{d}t\leq c_3^6.
\end{split}
\end{equation}

\subsection{Passing to the limit as $\eta\rightarrow 0 $}\label{sectionlim}
Now we consider the  systems \eqref{linear} when $\eta\rightarrow 0$ as follows:
\begin{equation}\label{19}
\begin{cases}
&\varphi_t+v\cdot \D  \varphi+\frac{\delta-1}{2} \tilde{\varphi}  \div v=0,\\
&A_0 W_t+ \sum_{j=1}^{3}A_j(V)\partial_j W+(\varphi^2+\eta^2)  \mathbb{L}(W)= \mathbb{Q}(V)H(\varphi)+\left(\begin{matrix}
0\\
a_1\D  \Phi
\end{matrix}\right)   , \\
&\Delta\Phi=\varphi^{\frac{2}{\delta-1}}-b,\\
&(\varphi, W,\Phi)|_{t=0}=\left(\varphi_0(x), W_0(x),\Phi_0(x)\right), \quad x \in \mathbb{T}^3, 
\end{cases}
\end{equation}

\begin{lemma} \label{32} Assume $\left(\varphi_0,W_0\right)$ satisfy \eqref{18}. Then there exists a time $T^*>0$  and a unique strong solution $(\varphi,W, \Phi)$ in $\left[0, T^*\right] \times \mathbb{T}^3$ to \eqref{19} such that
\begin{equation}\label{48}
\begin{aligned}
& \varphi \in C\left(\left[0, T^*\right] ; H^3\right), \quad  \phi \in C\left(\left[0, T^*\right] ; H^3\right), \\
& u \in C\left(\left[0, T^*\right] ; H^{s^{\prime}}\right) \cap L^{\infty}\left(\left[0, T^*\right] ; H^3\right), \quad s^{\prime} \in[2,3), \\
& \varphi \nabla^4 u \in L^2\left(\left[0, T^*\right] ; L^2\right), \quad  u_t \in C\left(\left[0, T^*\right] ; H^1\right) \cap L^2\left(\left[0, T^*\right] ; D^2\right),\\
& \Phi\in  C\left(\left[0, T^*\right] ; H^3\right)\cap L^\infty\left(\left[0, T^*\right] ; H^4\right).
\end{aligned}
\end{equation}
 Moreover, $(\varphi, W, \Phi)$ also satisfies the a priori estimates $\eqref{20}$.
\end{lemma}

\begin{proof}
We prove the existence, uniqueness and time-continuity in three steps.

\textbf{Step 1}. Existence.  Due to Lemma \ref{21} and the uniform estimates \eqref{20}, for every $\eta>0$, there exists a unique strong solution $\left(\varphi^\eta, W^\eta, \Phi^\eta \right)$ in $\left[0, T^*\right] \times \mathbb{T}^3$ to the linearized problem \eqref{linear} satisfying estimates \eqref{20}, where the time $T^*>0$ is  independent of $\eta$.

By virtue of the uniform estimates \eqref{20} independent of $\eta$ and  Lemma \ref{22}, we know that for any $R>0$, there exists a subsequence of solutions (still denoted by) $\left(\rho^\eta, \varphi^\eta,  W^\eta\right)$, which converges to a limit $( \varphi,   W)=( \varphi, \phi, u)$ in the following strong sense:
\begin{equation}\label{24}
\left( \varphi^\eta,  W^\eta\right) \rightarrow(\varphi,W) \quad \text { in } ~ C\left(\left[0, T^*\right] ; H^2\left(B_R\right)\right), \quad \text { as } ~ \eta \rightarrow 0 .
\end{equation}

Again by virtue of the uniform estimates \eqref{20} independent of $\eta$, we also know that there exists a subsequence of solutions (still denoted by) $\left(\varphi^\eta,  W^\eta, \Phi^\eta \right)$, which converges to $( \varphi,W,\Phi^\eta)$ in the following weak or weak--$\ast$ sense:
\begin{equation}\label{25}
\begin{aligned}
\left( \varphi^\eta,  W^\eta\right)\rightharpoonup(\rho, \varphi,W) & \text { weakly--$\ast$ in } L^{\infty}\left(\left[0, T^*\right] ; H^3\left(\mathbb{T}^3\right)\right), \\
\left( \varphi_t^\eta,  \phi_t^\eta\right)\rightharpoonup\left(\rho_t, \varphi_t,\phi_t \right) & \text { weakly--$\ast$ in } L^{\infty}\left(\left[0, T^*\right] ; H^2\left(\mathbb{T}^3\right)\right), \\
\Phi^\eta \rightharpoonup\Phi & \text { weakly--$\ast$ in } L^{\infty}\left(\left[0, T^*\right] ; H^4\left(\mathbb{T}^3\right)\right), \\
u_t^\eta \rightharpoonup u_t & \text { weakly--$\ast$ in } L^{\infty}\left(\left[0, T^*\right] ; H^1\left(\mathbb{T}^3\right)\right), \\
u_t^\eta \rightharpoonup u_t & \text { weakly in } L^2\left(\left[0, T^*\right] ; D^2\left(\mathbb{T}^3\right)\right),
\end{aligned}
\end{equation}
which, along with the lower semi-continuity of weak convergence, implies that $(\varphi, W,\Phi)$ also satisfies the corresponding estimates \eqref{20} except those of $\varphi \nabla^4 u$.

Combining the strong convergence in \eqref{24} and the weak convergence in \eqref{25}, we easily obtain that $(\rho,\varphi, W)$ also satisfies the local estimates \eqref{20} and
\begin{equation}\label{27}
\begin{aligned}
&\varphi^\eta \nabla^4 u^\eta\rightharpoonup\varphi \nabla^4 u \quad \text{weakly in}\quad  L^2\left(\left[0, T^*\right] \times \mathbb{T}^3\right).
\end{aligned}
\end{equation}

Now we want to show that $(\rho,\varphi, W)$ is a weak solution in the sense of distributions to the linearized problem \eqref{19}. Multiplying $\eqref{19}_2$ by test function $f(t, x)=$ $\left(f^1, f^2, f^3\right) \in C_c^{\infty}\left(\left[0, T^*\right) \times \mathbb{T}^3\right)$ on both sides, and integrating over $\left[0, T^*\right] \times \mathbb{T}^3$, we have
\begin{equation}\label{28}
\begin{aligned}
& \int_0^t \int_{\mathbb{T}^3} u^\eta \cdot f_t \mathrm{d} x \mathrm{d} s-\int_0^t \int_{\mathbb{T}^3}(v \cdot \nabla) u^\eta \cdot f \mathrm{d} x \mathrm{d} s-\int_0^t \int_{\mathbb{T}^3} \frac{2 A \gamma}{\gamma-1} \tilde{\phi} \nabla \phi^\eta f \mathrm{d} x \mathrm{d} s \\
=&-\int u_0 \cdot f(0, x)+\int_0^t \int_{\mathbb{T}^3} \Big((\varphi^\eta)^2+\eta^2\Big) L(u^\eta) f \mathrm{d} x \mathrm{d} s \\
&-\frac{\delta}{\delta-1}\int_0^t \int_{\mathbb{T}^3} Q(v) \cdot \nabla(\varphi^\eta)^2 f \mathrm{d} x \mathrm{d} s-a_1\int_0^t \int_{\mathbb{T}^3} \nabla \Phi^\eta f \mathrm{d} x \mathrm{d} s.
\end{aligned}
\end{equation}

Combining the strong convergence in \eqref{24} and the weak convergences in \eqref{25}--\eqref{27},  and letting $\eta \rightarrow 0$ in \eqref{28}, we have
\begin{equation}
\begin{aligned}
& \int_0^t \int_{\mathbb{T}^3} u \cdot f_t \mathrm{d} x \mathrm{d} s-\int_0^t \int_{\mathbb{T}^3}(v \cdot \nabla) u \cdot f \mathrm{d} x \mathrm{d} s-\frac{2 A \gamma}{\gamma-1} \int_0^t \int_{\mathbb{T}^3} \tilde{\phi} \nabla \phi f \mathrm{d} x \mathrm{d} s \\
=&-\int u_0 \cdot f(0, x)+\int_0^t \int_{\mathbb{T}^3} \varphi^2 L(u) f \mathrm{d} x \mathrm{d} s\\
&-\frac{\delta}{\delta-1}\int_0^t \int_{\mathbb{T}^3} Q(v) \cdot \nabla \varphi^2 f \mathrm{d} x \mathrm{d} s-a_1\int_0^t \int_{\mathbb{T}^3} \nabla \Phi f \mathrm{d} x \mathrm{d} s.
\end{aligned}
\end{equation}
Thus it is obvious that $(\varphi, W,\Phi)$ is a weak solution in the sense of distributions to the linearized problem \eqref{19}, satisfying the regularities
\begin{equation}\label{49}
\begin{aligned}
&(\varphi,\phi,\Phi) \in C\left(\left[0, T^*\right] ; H^3\right), \quad (\varphi_t, \phi_t)\in C\left(\left[0, T^*\right] ; H^2\right)  \\
& u \in C\left(\left[0, T^*\right] ; H^{s^{\prime}}\right) \cap L^{\infty}\left(\left[0, T^*\right] ; H^3\right), \quad s^{\prime} \in[2,3), \\
& \varphi \nabla^4 u \in L^2\left(\left[0, T^*\right] ; L^2\right), \quad  u_t \in C\left(\left[0, T^*\right] ; H^1\right) \cap L^2\left(\left[0, T^*\right] ; D^2\right).
\end{aligned}
\end{equation}

\textbf{Step 2}. Uniqueness. Let $\left(\varphi_1, W_1,\Phi_1 \right)$ and $\left(\varphi_2, W_2,\Phi_2 \right)$ be two solutions obtained in the above step. We denote
\begin{equation*}
 \bar{\varphi}=\varphi_1-\varphi_2,  \quad \bar{W}=W_1-W_2, \quad \bar{\Phi}=\Phi_1-\Phi_2.
\end{equation*}
Then from $\eqref{19}_1$ and $\eqref{19}_3$, we have
\begin{equation*}
  \bar{\varphi}_t+v \cdot \nabla \bar{\varphi}=0, \quad \Delta\bar{\Phi}=0,
\end{equation*}
which implies that $ \bar{\varphi}=0, \bar{\Phi}=0 $.
Let $\bar{W}=(\bar{\phi}, \bar{u})^{\top}$, from $\eqref{19}_3$ and $\varphi_1=\varphi_2$, we have
\begin{equation}\label{42}
A_0 \bar{W}_t+ A_1(V)  \bar{W}_x=- \varphi_1^2 \mathbb{L}(\bar{W}) .
\end{equation}
Then multiplying \eqref{42} by $\bar{W}$ on both sides, and integrating over $\mathbb{T}^3$, we have
\begin{equation}
\begin{aligned}
&\frac{1}{2}  \frac{\mathrm{d}}{\mathrm{d} t} \int \bar{W}^{\top} A_0 \bar{W}+a_1\alpha \left|\varphi_1  \nabla\bar{u}\right|_2^2 \\
 \leq~ & C| \nabla V|_{\infty}|\bar{W}|_2^2+\left|\nabla \varphi_1\right|_{\infty}|\bar{u}|_2\left|\varphi_1 \nabla \bar{u}\right|_2  \\
 \leq~ & \frac{a_1 \alpha}{10}\left|\varphi_1  \nabla\bar{u}\right|_2^2+Cc_2^2|\bar{W}|_2^2,
\end{aligned}
\end{equation}
which yields that
\begin{equation}
  \frac{\mathrm{d}}{\mathrm{d} t} | \bar{W}|_2^2+\left|\varphi_1 \nabla\bar{u}\right|_2^2\leq Cc_2^2|\bar{W}|_2^2.
\end{equation}
From Gronwall's inequality, we obtain that $\bar{W}=0$, which gives the uniqueness.

\textbf{Step 3}. Time-continuity. First for $\varphi$, via the regularities shown in \eqref{49} and the classical Sobolev embedding theorem, we have
\begin{equation}\label{43}
\varphi \in C\left(\left[0, T^*\right] ; H^2\right) \cap C\left(\left[0, T^*\right] ; \text { weak }-H^3\right) .
\end{equation}
Using the same arguments as in Lemma \ref{pr1}, we have
\begin{equation*}
\|\varphi(t)\|_3^2 \leq\left(\left\|\varphi_0\right\|_3^2+C \int_0^t\left(\|\tilde{\varphi}\|_3^2\|v\|_3^2+\left|\tilde{\varphi}  \nabla^4 v\right|_2^2\right) \mathrm{d} s\right) \exp \left(C \int_0^t\|v\|_3 \mathrm{~d} s\right),
\end{equation*}
which implies that
\begin{equation*}
\limsup_{t \rightarrow 0}\|\varphi(t)\|_3 \leq\left\|\varphi_0\right\|_3 .
\end{equation*}
Then according to Lemma \ref{104} and \eqref{43}, we know that $\varphi$ is right continuous at $t=0$ in $H^3$ space. From the reversibility on the time to equation $\eqref{19}_1$, we know
\begin{equation}\label{44}
\varphi \in C\left(\left[0, T^*\right] ; H^3\right).
\end{equation}
For $\varphi_t$, from
\begin{equation*}
\varphi_t=-v \cdot \nabla \varphi-\frac{\delta-1}{2} \tilde{\varphi}  \text{div}v,
\end{equation*}
we only need to consider the term $\tilde{\varphi} \text{div} v$. Due to
\begin{equation*}
\tilde{\varphi}  \text{div}v \in L^2\left(\left[0, T^*\right] ; H^3\right), \quad(\tilde{\varphi}  \text{div}v)_t \in L^2\left(\left[0, T^*\right] ; H^1\right),
\end{equation*}
and the Sobolev embedding theorem, we have
\begin{equation*}
\tilde{\varphi}  \text{div}v \in C\left(\left[0, T^*\right] ; H^2\right),
\end{equation*}
which implies that
\begin{equation*}
\varphi_t \in C\left(\left[0, T^*\right] ; H^2\right).
\end{equation*}
The similar arguments can be used to deal with the regularities of $ \phi$ and we can get that
\begin{equation}\label{200}
 \phi\in C\left(\left[0, T^*\right] ; H^3\right), \quad  \phi_t \in C\left(\left[0, T^*\right] ; H^2\right).
\end{equation}
Using the similar methods as in the proof of Lemma \ref{21}, we can get that $\Phi\in C\left(\left[0, T^*\right] ; H^3\right)$.

For velocity $u$, from the regularity shown in \eqref{49} and Sobolev's embedding theorem, we obtain that
\begin{equation}\label{45}
u \in C\left(\left[0, T^*\right] ; H^1\right) \cap C\left(\left[0, T^*\right] ; \text {weak}-H^2\right) .
\end{equation}
Then from Lemma \ref{Lem:2.4}, for any $s^{\prime} \in[2,3)$, we have
\begin{equation*}
\|u\|_{s^{\prime}} \leq C|u|_2^{1-\frac{s^{\prime}}{2}}\|u\|_2^{\frac{s^{\prime}}{2}} \text {. }
\end{equation*}
Together with the upper bound shown in \eqref{20} and the time continuity \eqref{45}, we have
\begin{equation}\label{46}
u \in C(\left[0, T^*\right] ; H^{s^{\prime}}) .
\end{equation}
Finally, we consider $u_t$. From equations $\eqref{19}_3$ we have
\begin{equation*}
u_t=-v \cdot \nabla u-\frac{2A  \gamma}{\gamma-1} \tilde{\phi}\nabla \phi-\alpha \varphi^2 L (u)+\frac{\delta}{\delta-1}Q(v)\cdot \nabla\varphi^2+\nabla \Phi.
\end{equation*}
From \eqref{49}, we have
\begin{equation*}
\begin{aligned}
 \varphi^2 L( u) \in L^2\left(\left[0, T^*\right] ; H^2\right), \quad \left(\varphi^2 L( u)\right)_t \in L^2\left(\left[0, T^*\right] ; L^2\right),
 \end{aligned}
\end{equation*}
which means that
\begin{equation}\label{47}
\varphi^2 L( u) \in C\left(\left[0, T^*\right] ; H^1\right).
\end{equation}

Combining \eqref{48},\eqref{44}--\eqref{200},\eqref{46} and \eqref{47}, we deduce that
\begin{equation}
u_t \in C\left(\left[0, T^*\right] ; H^1\right) .
\end{equation}  

Hence we complete the proof.
\end{proof}

\subsection{Proof of Theorem \ref{29}}
Our proof is based on the classical iteration scheme and the existence results for the linearized problem obtained in Section \ref{sectionlim}. Like in Section \ref{estimates}, we define constants $c_0$ and $c_i(i=1,2,3)$, and assume that
$$
1+\left\|\varphi_0\right\|_3+\left\|W_0\right\|_3 \leq c_0.
$$

Let $\left(\varphi^0, W^0=\left(\phi^0, u^0\right)\right)$, with the regularities
$$
\begin{aligned}
& \varphi^0 \in C\left(\left[0, T^{*}\right] ; H^3\right),  \quad \phi^0 \in C\left(\left[0, T^{*}\right] ; H^3\right), \quad \varphi^0 \nabla^4 u^0 \in L^2\left(\left[0, T^{*}\right] ; L^2\right),\\
& u^0 \in C(\left[0, T^{*}\right] ; H^{s^{\prime}}) \cap L^{\infty}\left(\left[0, T^{*}\right] ; H^3\right) \text { for any } s^{\prime} \in[2,3),
\end{aligned}
$$
be the solution to the problem
\begin{equation}
\left\{\begin{array}{l}
X_t+u_0 \cdot \nabla X=0 \quad \text { in }(0,+\infty) \times \mathbb{T}^3, \\
Y_t+u_0 \cdot \nabla Y=0 \quad \text { in }(0,+\infty) \times \mathbb{T}^3, \\
Z_t-X^2 \triangle Z=0 \quad \text { in }(0,+\infty) \times \mathbb{T}^3, \\
\left.(X, Y, Z)\right|_{t=0}=\left(\varphi_0, \phi_0, u_0\right) \quad \text { in } \mathbb{T}^3. 
\end{array}\right.
\end{equation}

We take a time $T^{**} \in\left(0, T^{*}\right]$ small enough such that
\begin{align}
\sup\limits_{0 \leq t \leq T^{**}}\left(|| \varphi^0(t)\left\|_1^2+|| \phi^0(t)\right\|_1^2+|| u^0(t) \|_1^2\right)+\int_0^{T^{**}} \left|\varphi^0 \nabla^2 u^0\right| \mathrm{d} t\leq~& c_1^2, \nonumber\\
\sup\limits_{0 \leq t \leq T^{**}}\left(\left|\varphi^0(t)\right|_{D^2}^2+\left|\phi^0(t)\right|_{D^2}^2+\left|u^0(t)\right|_{D^2}^2\right)+\int_0^{T^{* *}} \left|\varphi^0 \nabla^3 u^0\right| \mathrm{d} t\leq~& c_2^2, \\
\sup\limits_{0 \leq t \leq T^{** }}\left(\left|\varphi^0(t)\right|_{D^3}^2+\left|\phi^0(t)\right|_{D^3}^2+\left|u^0(t)\right|_{D^3}^2\right)+\int_0^{T^{**}} \left|\varphi^0 \nabla^4 u^0\right| \mathrm{d} t\leq~& c_3^2.\nonumber
\end{align}
\begin{proof}
 We prove the existence, uniqueness and time-continuity in three steps. 

\textbf{Step 1}. Existence. Let $(\tilde{\varphi}, \tilde{\phi},v)=\left(\varphi^0, \phi^0, u^0\right)$, we define $\left(\rho^1, \varphi^1, W^1,\Phi^1 \right)$ as a strong solution to problem \eqref{19}. Then we construct approximate solutions
\begin{equation*}
\left( \varphi^{k+1}, W^{k+1},\Phi^{k+1}\right)=\left( \varphi^{k+1},\phi^{k+1}, u^{k+1},\Phi^{k+1}\right)
\end{equation*}
inductively, by assuming that $\left( \varphi^k, W^k,\Phi^k \right)$ was defined for $k \geq 1$, let $\left(\varphi^{k+1}, W^{k+1},\Phi^{k+1}\right)$ be the unique solution to problem \eqref{19} with  $( \tilde{\varphi}, \tilde{\phi},v)$  replaced by $\left( \varphi^k, W^k\right)$ as follows:
\begin{equation}\label{31}
\begin{cases}
&\varphi^{k+1}_t+u^k\cdot \nabla \varphi^{k+1}+\frac{\delta-1}{2} \varphi^k  \text{div}u^k=0, \\
&A_0 W^{k+1}_t+ \sum_{j=1}^{3}A_j(W^k)\partial_j W^{k+1}+(\varphi^{k+1})^2 \mathbb{L}(W^{k+1})=\mathbb{Q}(W^k)\cdot H(\varphi^{k+1})+ a_1\left(\begin{matrix}
0\\
\D  \Phi^{k+1}
\end{matrix}\right)   , \\
&\Delta \Phi^{k+1}=(\varphi^{k+1})^{\frac{2}{\delta-1}}-b,\\
&( \varphi^{k+1}, W^{k+1},\Phi^{k+1})|_{t=0}=\left(\varphi_0(x), W_0(x),\Phi_0 \right), \quad x \in \mathbb{T}^3.
\end{cases}
\end{equation}

It follows from Lemma \ref{32} that the sequence $\left(\varphi^k, W^k, \Phi^k \right)$ satisfies the uniform a priori estimates \eqref{20} for $0 \leq t \leq T^{* *}$. Then, from \eqref{31}, we can obtain that
\begin{equation}\label{33}
\left\{\begin{aligned}
&\bar{\varphi}_t^{k+1}+u^k \cdot \nabla \bar{\varphi}^{k+1}+\bar{u}^k \cdot \nabla \varphi^k+\frac{\delta-1}{2}(\bar{\varphi}^k \operatorname{div} u^{k-1}+\varphi^k \operatorname{div} \bar{u}^k)=0, \\
&A_0 \bar{W}_t^{k+1}+\sum_{j=1}^3 A_j(W^k) \partial_j \bar{W}^{k+1}+(\varphi^{k+1})^2 \mathbb{L}(\bar{W}^{k+1}) \\
=&\sum_{j=1}^3 A_j(\bar{W}^k) \partial_j W^k- \bar{\varphi}^{k+1}(\varphi^{k+1}+\varphi^k) \mathbb{L}(W^k)+\mathbb{Q}(W^k)\cdot(\mathbb{H}(\varphi^{k+1})-\mathbb{H}(\varphi^k))\\
&+\mathbb{Q}(\bar{W}^k)\cdot \mathbb{H}(\varphi^{k})+ a_1\left(\begin{matrix}
0\\
\D  \bar{\Phi}^{k+1}
\end{matrix}\right),\\
&\triangle \bar{\Phi}^{k+1}=(\varphi^{k+1})^{\frac{2}{\delta-1}}-(\varphi^k)^{\frac{2}{\delta-1}}.
\end{aligned}\right.
\end{equation}

First, we consider $|\bar{\varphi}^{k+1}|_2$. Multiplying $\eqref{33}_1$ by $2 \bar{\varphi}^{k+1}$ and integrating over $\mathbb{T}^3$, one has
$$
\begin{aligned}
\frac{\mathrm{d}}{\mathrm{d}t}|\bar{\varphi}^{k+1}|_2^2= & -2 \int(u^k \cdot \nabla \bar{\varphi}^{k+1}+\bar{u}^k \cdot \nabla \varphi^k  +\frac{\delta_1-1}{2}\left(\bar{\varphi}^k \operatorname{div} u^{k-1}+\varphi^k \operatorname{div} \bar{u}^k\right) \bar{\varphi}^{k+1} \\
\leq & C|\nabla u^k|_{\infty}|\bar{\varphi}^{k+1}|_2^2+C|\bar{\varphi}^{k+1}|_2(|\bar{u}^k|_2|\nabla \varphi^k|_{\infty}  +|\bar{\varphi}^k|_2|\nabla u^{k-1}|_{\infty}+|\varphi^k \operatorname{div} \bar{u}^k|_2),
\end{aligned}
$$
which means that 
\begin{equation}\label{37}
\frac{\mathrm{d}}{\mathrm{d}t}|\bar{\varphi}^{k+1}(t)|_2^2 \leq C_\nu|\bar{\varphi}^{k+1}(t)|_2^2+\nu\left(|\bar{u}^k(t)|_2^2+|\bar{\varphi}^k(t)|_2^2+|\varphi^k \operatorname{div} \bar{u}^k(t)|_2^2\right)
\end{equation}
with $C_\nu=C\left(1+\nu^{-1}\right)$ and   $0<\nu \leq \frac{1}{10}$  is a constant.

Furthermore, from $\eqref{33}_3$ and Lagrange’s mean value theorem, we can easily deduce that
\begin{equation}\label{202}
    \triangle \bar{\Phi}^{k+1}=\frac{2}{\delta-1} (\theta^{k+1})^{\frac{3-\delta}{\delta-1}} \bar{\varphi}^{k+1}\leq C \bar{\varphi}^{k+1},
\end{equation}
due to $\delta\in (1,2)\cup \{3 \}$, where $\theta^{k+1}$ is between $\varphi^{k+1}$ and $\varphi^{k}.$ 

It follows from \eqref{202} and \eqref{37} that 
\begin{equation}\label{201}
\begin{aligned}
& \frac{\mathrm{d}}{\mathrm{d} t} \int_{\mathbb{T}^3}\left|\Delta \bar{\Phi}^{k+1}\right|^2 \mathrm{~d} x\leq C \frac{\mathrm{d}}{\mathrm{d} t} \left|\bar{\varphi}^{k+1}\right|_2^2  \leq C_\nu|\bar{\varphi}^{k+1}(t)|_2^2+\nu\left(|\bar{u}^k(t)|_2^2+|\bar{\varphi}^k(t)|_2^2+|\varphi^k \operatorname{div} \bar{u}^k(t)|_2^2\right), \\
& \left\|\bar{\Phi}^{k+1}\right\|_{2} \leq C\left|\bar{\varphi}^{k+1}\right|_{2} .
\end{aligned}
\end{equation}

Next, we consider $|\bar{W}^{k+1}|_2$. Multiplying $\eqref{31}_3$ by $2 \bar{W}^{k+1}$ and integrating over $\mathbb{T}^3$, we obtain that
\begin{equation}\label{35}
\begin{aligned}
&\frac{\mathrm{d}}{\mathrm{d}t}  \int(\bar{W}^{k+1})^{\top} A_0 \bar{W}^{k+1}+2 a_1  \alpha|\varphi^{k+1} \nabla \bar{u}^{k+1}|_2^2+2a_1 (\alpha+\beta)  |\varphi^{k+1} \operatorname{div} \bar{u}^{k+1}|_2^2 \\
\leq&  \int(\bar{W}^{k+1})^{\top} \operatorname{div} A(W^k) \bar{W}^{k+1}+\int \sum_{j=1}^3(\bar{W}^{k+1})^{\top} A_j(\bar{W}^k) \partial_j W^k \\
& -2 a_1 \frac{\delta}{\delta-1}  \int \nabla(\varphi^{k+1})^2 \cdot Q(\bar{u}^{k+1}) \cdot \bar{u}^{k+1}  -2 a_1  \int(\bar{\varphi}^{k+1}(\varphi^{k+1}+\varphi^k) \cdot L(u^k))\cdot \bar{u}^{k+1} \\
& -2 a_1 \frac{\delta-1}{\delta} \int \nabla(\bar{\varphi}^{k+1}(\varphi^{k+1}+\varphi^k)) \cdot Q(u^k)\cdot \bar{u}^{k+1}\\
& +2 a_1 \frac{\delta-1}{\delta}\int \nabla(\varphi^k)^2 \cdot(  Q(u^k)-Q( u^{k-1})) \cdot \bar{u}^{k+1}+ a_1 \int \nabla \bar{\Phi}\cdot \bar{u}^{k+1}\\
=&\sum_{i=1}^{7} \mathrm{J}_i .
\end{aligned}
\end{equation}
It follows from  Lemma \ref{GN}, \eqref{13},  H\"older's inequality, Young's inequality  and $\eqref{201}_2$ that 
\begin{align}
\mathrm{J}_1=~ & \int\left(\bar{W}^{k+1}\right)^{\top} \operatorname{div} A(W^k) \bar{W}^{k+1} \leq C|\nabla W^k|_{\infty}|\bar{W}^{k+1}|_2^2 \leq C|\bar{W}^{k+1}|_2^2, \nonumber\\
\mathrm{J}_2=~ & \int \sum_{j=1}^3 A_j\left(\bar{W}^k\right) \partial_j W^k \cdot \bar{W}^{k+1} \nonumber \\
\leq~ & C|\nabla W^k|_{\infty}|\bar{W}^k|_2|\bar{W}^{k+1}|_2 \leq C \nu^{-1}|\bar{W}^{k+1}|_2^2+\nu|\bar{W}^k|_2^2, \nonumber\\
\mathrm{J}_3=~ & -2 a_1 \frac{\delta}{\delta-1}  \int \nabla(\varphi^{k+1})^2 \cdot Q(\bar{u}^{k+1}) \cdot \bar{u}^{k+1}\nonumber\\
\leq~ & C |\nabla \varphi^{k+1}|_{\infty}|\varphi^{k+1} \nabla \bar{u}^{k+1}|_2|\bar{u}^{k+1}|_2 \label{34} \\
\leq~ & C |\bar{W}^{k+1}|_2^2+\frac{a_1  \alpha}{10}|\varphi^{k+1} \nabla \bar{u}^{k+1}|_2^2, \nonumber \\
\mathrm{J}_4=~ & -2 a_1  \int(\bar{\varphi}^{k+1}(\varphi^{k+1}+\varphi^k) \cdot L(u^k))\cdot \bar{u}^{k+1}\nonumber\\
\leq~ & C|\bar{\varphi}^{k+1}|_2|\bar{u}^{k+1}|_2|\varphi^k \nabla^2 u^k|_\infty+C|\bar{\varphi}^{k+1}|_2|\varphi^{k+1}\nabla \bar{u}^{k+1}|_2 \nonumber\\
\leq~ & C|\bar{\varphi}^{k+1}|_2^2+C(1+|\varphi^{k+1}\nabla^4 u^k|_2)|\bar{u}^{k+1}|_2^2+\frac{\alpha}{10}|\varphi^{k+1}\nabla \bar{u}^{k+1}|_2^2, \nonumber\\
\mathrm{J}_5=~& -2 a_1 \frac{\delta-1}{\delta} \int \nabla(\bar{\varphi}^{k+1}(\varphi^{k+1}+\varphi^k)) \cdot Q(u^k)\cdot \bar{u}^{k+1}\nonumber\\
\leq~ & C|\nabla^2 u^k|_6|\bar{\varphi}^{k+1}|_2|\varphi^{k+1}\bar{u}^{k+1}|_3+C|\bar{\varphi}^{k+1}|_2|\varphi^k\nabla^2u^k|_\infty |\bar{u}^{k+1}|_2 \nonumber\\
&+C|\bar{\varphi}^{k+1}|_2|\nabla u^k|_\infty |\varphi^{k+1}\nabla \bar{u}^{k+1}|_2+C|\bar{\varphi}^{k+1}|_2|\varphi^{k+1}\nabla \bar{u}^{k+1}|_2|\nabla u^k|_\infty \nonumber\\
&+C |\bar{\varphi}^{k+1}|_2^2 |\nabla u^k|_\infty|\nabla \bar{u}^{k+1}|_\infty \nonumber\\
\leq~ & C |\bar{\varphi}^{k+1}|_2^2+C(1+|\varphi^k\nabla^4 u^k|_2^2)|\bar{W}^{k+1}|_2^2+\frac{a_1\alpha}{10}|\varphi^{k+1}\nabla \bar{u}^{k+1}|_2^2,\nonumber\\
\mathrm{J}_6=~&2 a_1 \frac{\delta-1}{\delta}\int \nabla(\varphi^k)^2 \cdot(  Q(u^k)-Q( u^{k-1})) \cdot \bar{u}^{k+1} \nonumber\\
\leq~ & C|\nabla \varphi^k|_\infty|\varphi^k\nabla\bar{u}^k|_2|\bar{u}^{k+1}|_2\leq C\nu^{-1}|\bar{W}^{k+1}|_2^2+\nu|\varphi^k\nabla\bar{u}^k|_2^2, \nonumber \\
\mathrm{J}_{7}=~&  a_1 \int\nabla \bar{\Phi}\cdot \bar{u}^{k+1}\leq C|\nabla \Phi|_2 |\bar{u}^{k+1}|_2\leq C |\nabla \Phi|_2^2+ C|\bar{u}^{k+1}|_2^2 \nonumber\\
\leq~ & C \left|\bar{\varphi}^{k+1}\right|_{2}^2+C|\bar{u}^{k+1}|_2^2.\nonumber
\end{align}
Then, from \eqref{35} and \eqref{34},  it yields that 
\begin{equation}\label{36}
\begin{aligned}
&\frac{\mathrm{d}}{\mathrm{d}t} \int(\bar{W}^{k+1})^{\top} A_0 \bar{W}^{k+1}+a_1  \alpha|\varphi^{k+1} \nabla \bar{u}^{k+1}|_2^2 \\
\leq~ &  C\left(\nu^{-1}+|\varphi^k \nabla^4 u^k|_2^2\right)|\bar{W}^{k+1}|_2^2+C |\bar{\varphi}^{k+1}|_2^2 +\nu\left(|\varphi^k \nabla\bar{u}^k|_2^2+|\bar{\varphi}^k|_2^2+|\bar{W}^k|_2^2\right) .
\end{aligned}
\end{equation}

We denote
\begin{equation*}
  S^{k+1}(t)=\sup\limits_{s\in[0,t]}|\bar{W}^{k+1}(s)|_2^2+\sup\limits_{s\in[0,t]}|\bar{\varphi}^{k+1}(s)|_2^2+\sup\limits_{s\in[0,t]}|\triangle \bar{\Phi}^{k+1}(s)|_2^2.
\end{equation*}

From \eqref{37}, \eqref{201} and \eqref{36}, we have
\begin{equation}
\begin{aligned}
& \frac{\mathrm{d}}{\mathrm{d}t} \int\Big((\bar{W}^{k+1})^{\top} A_0 \bar{W}^{k+1}+|\bar{\varphi}^{k+1}(s)|_2^2+|\triangle \bar{\Phi}^{k+1}(s)|_2^2\Big)+|\varphi^{k+1} \nabla \bar{u}^{k+1}|_2^2 \\
\leq~ &  E^k_\nu(|\bar{W}^{k+1}|_2^2+|\bar{\varphi}^{k+1}|_2^2 ) +\nu\left(|\varphi^k \nabla\bar{u}^k|_2^2+|\bar{\varphi}^k|_2^2+|\bar{W}^k|_2^2\right)
\end{aligned}
\end{equation}
for some $E^k_\nu$ such that $\int_{0}^{t} E^k_\nu ds\leq C+C(1+\frac{1}{\nu})t$. According to Gronwall's inequality, one has
\begin{equation}
\begin{aligned}
&S^{k+1}+\int_{0}^{t}|\varphi^{k+1}\nabla\bar{u}^{k+1}|_2^2\mathrm{d}s\\
\leq~ & C\nu \int_{0}^{t}(|\varphi^k \nabla\bar{u}^k|_2^2+|\bar{\varphi}^k|_2^2+|\bar{W}^k|_2^2)\mathrm{d}s\exp(C+C(1+\frac{1}{\nu})t)\\
\leq~ & \Big(C\nu\int_{0}^{t}(|\varphi^k \nabla\bar{u}^k|_2^2)\mathrm{d}s+Ct\nu\sup\limits_{s\in[0,t]}[|\bar{W}^k|_2^2+|\bar{\varphi}^k|_2^2]\Big)\exp(C+C(1+\frac{1}{\nu})t).
\end{aligned}
\end{equation}

We can choose $\nu_0>0$ and $T_{*}\in(0,\min(1,T^{**}))$ small enough such that 
\begin{equation*}
  C\nu_0\exp C\leq \frac{1}{8},\quad \exp(C(1+\frac{1}{\nu})T_*)\leq 2,
\end{equation*}
which yields that
\begin{equation}\label{40}
 \sum_{k=1}^{+\infty} \Big(S^{k+1}(T_*)+\int_{0}^{T_*}(|\varphi^{k+1} \nabla \bar{u}^{k+1}|_2^2\Big)\leq C< +\infty.
\end{equation}

It follows from \eqref{20} and \eqref{40} that $(\varphi^k,W^k,\Phi^k)$ converges to a limit $(\varphi,W,\Phi)$ in the following strong sense:
\begin{equation}
(\varphi^k, W^k, \Phi^k)\rightarrow (\varphi, W, \Phi)\quad \text{in} \quad L^\infty([0,T_*];H^2(\mathbb{T}^3).
\end{equation}

Due to the local estimates \eqref{20} and the lower-continuity of norm for weak or weak* convergence, we know that $(\varphi, W,\Phi)$ satisfies the estimates \eqref{20}. According to the strong convergence in \eqref{40}, we can  show that $(\varphi, W, \Phi)$ is a weak solution of \eqref{4} in the sense of distribution with the regularities:
\begin{equation}
\begin{aligned}
& \varphi \in L^{\infty}\left(\left[0, T^*\right] ; H^3\right), \varphi_t \in L^{\infty}\left(\left[0, T^*\right] ; H^2\right), \phi \in L^{\infty}\left(\left[0, T^*\right] ; H^3\right), \phi_t \in L^{\infty}\left(\left[0, T^*\right] ; H^2\right),\\
&\Phi \in L^\infty\left(\left[0, T^*\right] ; H^4\right),     u \in L^{\infty}\left(\left[0, T^*\right] ; H^3\right),  \varphi \nabla^4 u \in L^2\left(\left[0, T^*\right] ; L^2\right), \\
&u_t \in L^{\infty}\left(\left[0, T^*\right] ; H^1\right) \cap L^2\left(\left[0, T^*\right] ; D^2\right) .
\end{aligned}
\end{equation}

Thus the existence of strong solutions is proved.

\textbf{Step 2}. Uniqueness and time-continuity. It can be obtained via the same arguments used in the proof of Lemma \ref{32}. 
\end{proof}

\subsection{Proof of Theorem \ref{main theory}}\label{section 3.6}
\begin{proof}
Now we are ready to prove the Theorem \ref{main theory} and the proof is divided into two steps.

\textbf{Step 1}. Existence of regular solutions. First, for the initial assumption \eqref{203}, it follows from Theorem \ref{29} that there exists a positive time $T_*$  such that the problem \eqref{4} has a unique strong solution $(\varphi, \phi, u)$ in $\left[0, T_*\right] \times \mathbb{T}^3$ satisfying the regularities in \eqref{53}, which means that
$$
\left(\rho^{\frac{\delta-1}{2}}, \rho^{\frac{\gamma-1}{2}}\right)=(\varphi, \phi) \in C^1\left(\left(0, T_*\right) \times \mathbb{T}^3 \right), \quad \text { and } \quad(u, \nabla u) \in C\left(\left(0, T_*\right) \times \mathbb{T}^3 \right) .
$$

Noticing that $\rho=\varphi^{\frac{2}{\delta-1}}$ and $\frac{2}{\delta-1} \geq 1$, it is easy to show that
\begin{equation}\label{301}
\rho \in C^1\left(\left(0, T_*\right) \times \mathbb{T}^3 \right) .
\end{equation}

Second, the system $\eqref{4}_2$ for $W=(\phi, u)$ could be written as
\begin{equation}\label{204}
\left\{\begin{array}{l}
\phi_t+u \cdot \nabla \phi+\frac{\gamma-1}{2} \phi \operatorname{div} u=0, \\
u_t+u \cdot \nabla u+\frac{A \gamma}{\gamma-1} \nabla \phi^2+ \varphi^2 L u=\nabla \varphi^2 \cdot Q(u)+\nabla \Phi .
\end{array}\right.
\end{equation}

Multiplying $\eqref{204}_1$ by $\frac{\partial \rho}{\partial \phi}(t, x)=\frac{2}{\gamma-1} \phi^{\frac{3-\gamma}{\gamma-1}}(t, x) \in C\left(\left(0, T_*\right) \times\mathbb{T}^3 \right)$ on both sides, we get the continuity equation in $\eqref{CNSP}_1$ :
$$
\rho_t+u \cdot \nabla \rho+\rho \operatorname{div} u=0 .
$$

Multiplying $\eqref{204}_2$ by $\phi^{\frac{2}{\gamma-1}}=\rho(t, x) \in C^1\left(\left(0, T_*\right) \times \mathbb{T}^3 \right)$ on both sides, we get the momentum equations in $\eqref{CNSP}_2$ :
$$
\rho u_t+\rho u \cdot \nabla u+\nabla P=\operatorname{div}\left(\mu(\rho)\left(\nabla u+(\nabla u)^{\top}\right)+\lambda(\rho) \operatorname{div} u I_3\right)+\rho \nabla \Phi .
$$

Finally, recalling that $\rho$ can be represented by the formula
$$
\rho(t, x)=\rho_0(U(0, t, x)) \exp \left(\int_0^t \operatorname{div} u(s, U(s, t, x)) \mathrm{d} s\right),
$$
where $U \in C^1\left(\left[0, T_*\right] \times\left[0, T_*\right] \times \mathbb{T}^3\right)$ is the solution to the initial value problem
$$
\left\{\begin{array}{l}
\frac{d}{d s} U(t, s, x)=u(s, U(s, t, x)), \quad 0 \leqq s \leqq T_*, \\
U(t, t, x)=x, \quad 0 \leqq t \leqq T_*, \quad x \in \mathbb{T}^3,
\end{array}\right.
$$
it is obvious that
$$
\rho(t, x) \geq 0, \quad \forall(t, x) \in\left(0, T_*\right) \times \mathbb{T}^3 .
$$

That is to say, $(\rho, u,\Phi)$ satisfies the problem \eqref{CNSP} in the sense of distributions, and has the regularities shown in Definition \ref{220}, which means that the Cauchy problem \eqref{CNSP}-\eqref{data} has a unique regular solution $(\rho, u,\Phi)$.

\textbf{Step 2}. The smoothness of regular solutions. Now we will show that the regular solution that we obtained in the above step is indeed a classcial one in positive time $\left(0, T_*\right]$.

Due to the definition of regular solution and the classical Sobolev imbedding theorem, we immediately know that
$$
\left(\rho, \nabla \rho, \rho_t, u, \nabla u,\nabla \Phi \right) \in C\left(\left[0, T_*\right] \times \mathbb{T}^3 \right) .
$$

Now we only need to prove that
$$
\left(u_t, \text{div} \mathbb{S}\right) \in C\left(\left(0, T_*\right] \times \mathbb{T}^3\right) .
$$

Next, we first give the continuity of $u_t$. We differentiate $\eqref{204}_2$ with respect to $t$ :
\begin{equation}\label{205} 
u_{t t}+ \varphi^2 L u_t=-\left(\varphi^2\right)_t L u-(u \cdot \nabla u)_t-\frac{A \gamma}{\gamma-1} \nabla\left(\phi^2\right)_t+\left(\nabla \varphi^2 \cdot Q(u)\right)_t+\nabla \Phi_t,
\end{equation}
which, along with \eqref{53}, easily implies that
\begin{equation}\label{214}
    u_{t t} \in L^2\left(\left[0, T_*\right] ; L^2\right) .
\end{equation}
Applying $\partial_x^\zeta(|\zeta|=2)$ to \eqref{205}, multiplying the resulting equations by $\partial_x^\zeta u_t$ and integrating over
$\mathbb{T}^3$, we obtain
\begin{equation}\label{209}
\begin{aligned}
& \frac{1}{2} \frac{d}{d t}\left|\partial_x^\zeta u_t\right|_2^2+\alpha \left|\varphi \nabla \partial_x^\zeta u_t\right|_2^2+(\alpha+\beta) \left|\varphi \operatorname{div} \partial_x^\zeta u_t\right|_2^2 \\
=& \int\left(- \nabla \varphi^2 \cdot Q\left(\partial_x^\zeta u_t\right)-\left(\partial_x^\zeta\left(\varphi^2 L u_t\right)-\varphi^2 L \partial_x^\zeta u_t\right)\right) \cdot \partial_x^\zeta u_t \\
& +\int\left(- \partial_x^\zeta\left(\left(\varphi^2\right)_t L u\right)-\partial_x^\zeta(u \cdot \nabla u)_t-\frac{A \gamma}{\gamma-1} \partial_x^\zeta \nabla\left(\phi^2\right)_t\right) \cdot \partial_x^\zeta u_t \\
&+\int \partial_x^\zeta\left( \nabla \varphi^2 \cdot Q(u)\right)_t \cdot \partial_x^\zeta u_t +\int \nabla \partial_x^\zeta \Phi_t\cdot \partial_x^\zeta u_t \\
\triangleq &\sum_{i=1}^{7} P_i .
\end{aligned}
\end{equation}
It follows from the H\"older's inequality, Lemma \ref{GN} and Young's inequality that
\begin{align}
P_1=~& \int\left(-\nabla \varphi^2 \cdot Q\left(\partial_x^\zeta u_t\right)\right) \cdot \partial_x^\zeta u_t \nonumber\\
\leq~& C \left|\varphi \nabla^3 u_t\right|_2\left|\nabla^2 u_t\right|_2|\nabla \varphi|_{\infty} \leq \frac{\alpha }{20}\left|\varphi \nabla^3 u_t\right|_2^2+C \left|u_t\right|_{D^2}^2,  \nonumber\\
P_2=~&\int-\left(\partial_x^\zeta\left(\varphi^2 L u_t\right)-\varphi^2 L \partial_x^\zeta u_t\right) \cdot \partial_x^\zeta u_t  \nonumber\\
\leq~& C \left(\left|\varphi \nabla^3 u_t\right|_2|\nabla \varphi|_{\infty}+|\nabla \varphi|_{\infty}^2\left|u_t\right|_{D^2}+\left|\nabla^2 \varphi\right|_3\left|\varphi \nabla^2 u_t\right|_6\right)\left|u_t\right|_{D^2} \label{210} \\
\leq~& \frac{\alpha }{20}\left|\varphi \nabla^3 u_t\right|_2^2+C\left|u_t\right|_{D^2}^2,  \nonumber \\
P_3=~&\int- \partial_x^\zeta\left(\left(\varphi^2\right)_t L u\right) \cdot \partial_x^\zeta u_t  \nonumber \\
\leq~& C \left(\left|\nabla^2 \varphi\right|_3|L u|_6\left|\varphi_t\right|_{\infty}\left|u_t\right|_{D^2}+\left|\varphi \nabla^2 u_t\right|_6\left|\varphi_t\right|_{D^2}|L u|_3\right. \nonumber \\
& +|\nabla \varphi|_{\infty}\left|\nabla \varphi_t\right|_6|L u|_3\left|u_t\right|_{D^2}+\left|\varphi \nabla^3 u\right|_6\left|\nabla \varphi_t\right|_3\left|u_t\right|_{D^2} \nonumber \\
& \left.+|\nabla \varphi|_{\infty}\left|\varphi_t\right|_{\infty}\left|\nabla^3 u\right|_2\left|u_t\right|_{D^2}+\left|\varphi_t\right|_{\infty}\left|u_t\right|_{D^2}\left|\varphi \nabla^4 u\right|_2\right) \nonumber \\
\leq~& \frac{\alpha }{20}\left|\varphi \nabla^3 u_t\right|_2^2+C \left|u_t\right|_{D^2}^2+C \left|\varphi \nabla^4 u\right|_2^2+C ,  \nonumber\\
P_{4}=~&\int-\partial_x^\zeta(u \cdot \nabla u)_t \cdot \partial_x^\zeta u_t  \nonumber\\
\leq~& C\left(\left\|u_t\right\|_1+\left|u_t\right|_{D^2}\right)\|u\|_3-\int(u \cdot \nabla) \partial_x^\zeta u_t \cdot \partial_x^\zeta u_t  \nonumber\\
\leq~& C+C\left|u_t\right|_{D^2}^2+C|\nabla u|_{\infty}\left|\partial_x^\zeta u_t\right|_2^2 \leq C+C\left|u_t\right|_{D^2}^2,  \nonumber\\
P_{5}=~&\int-\frac{A \gamma}{\gamma-1} \partial_x^\zeta \nabla\left(\phi^2\right)_t \cdot \partial_x^\zeta u_t  \nonumber\\
\leq~& C\Big(\left|\nabla^2 \phi_t\right|_2\left|\phi \nabla^3 u_t\right|_2+\left|\phi_t\right|_{\infty}\left|\nabla^3 \phi\right|_2\left|\nabla^2 u_t\right|_2  \nonumber\\
& +C\left|\nabla^2 \phi\right|_6\left|\nabla \phi_t\right|_3\left|\nabla^2 u_t\right|_2+\left|\nabla^2 \phi_t\right|_2|\nabla \phi|_{\infty}\left|\nabla u_t\right|_2\Big)  \nonumber\\
\leq~& \frac{\alpha}{20}|\phi\nabla^3 u_t|_2^2+C(1+|\partial_x^\zeta u_t|^2), \nonumber\\
P_6=~&\int  \partial_x^\zeta\left(\nabla \varphi^2 \cdot Q(u)\right)_t \cdot \partial_x^\zeta u_t \nonumber \\
\leq~& \frac{\alpha}{20}|\varphi \nabla^3 u_t|_2^2 +C|\partial_x^\zeta u_t|_2^2+C|\varphi\nabla^4  u|_2^2+C, \nonumber\\
P_7=~&\int \nabla \partial_x^\zeta \Phi_t\cdot \partial_x^\zeta u_t \nonumber\\
\leq~& C|\partial_x^\zeta u_t|_2|\nabla^3 \Phi_t|_2 \nonumber\\
\leq~& C|\partial_x^\zeta u_t|_2^2+C.  \nonumber
\end{align}

It follows from $\eqref{209}-\eqref{210}$ that
\begin{equation}\label{211}
\frac{1}{2} \frac{d}{d t}\left|\nabla^2 u_t\right|_{2}^2+\frac{\alpha}{2}\left|\varphi \nabla^3 u_t\right|_2^2 \leq C \left|\nabla^2 
 u_t\right|_2^2+C \left|\varphi \nabla^4 u\right|_2^2+C.
\end{equation}

Then multiplying both sides of \eqref{211} with $t$ and integrating over $[\tau, t]$ for any $\tau \in(0, t)$, one gets
\begin{equation}\label{212}
t\left|u_t\right|_{D^2}^2+\int_\tau^t s\left|\varphi \nabla^3 u_t\right|_2^2 \mathrm{~d} s \leq C \tau\left|u_t(\tau)\right|_{D^2}^2+C(1+t) .
\end{equation}

According to the definition of the regular solution, we know that
$$
\nabla^2 u_t \in L^2\left(\left[0, T_*\right] ; L^2\right),
$$
which, along with Lemma \ref{221}, implies that there exists a sequence $s_k$ such that
$$
s_k \rightarrow 0, \quad \text { and } s_k\left|\nabla^2 u_t\left(s_k, \cdot\right)\right|_2^2 \rightarrow 0, \quad \text { as } k \rightarrow+\infty .
$$

Then, letting $\tau=s_k \rightarrow 0$ in \eqref{212}, we have
$$
t\left|u_t\right|_{D^2}^2+\int_0^t s\left|\varphi \nabla^3 u_t\right|_2^2 \mathrm{~d} s \leq C(1+t) \leq C,
$$

so we have
\begin{equation}\label{213}
t^{\frac{1}{2}} u_t \in L^{\infty}\left(\left[0, T_*\right] ; H^2\right) .
\end{equation}

Based on the classical Sobolev imbedding theorem
\begin{equation}\label{215}
L^{\infty}\left([0, T] ; H^1\right) \cap W^{1,2}\left([0, T] ; H^{-1}\right) \hookrightarrow C\left([0, T] ; L^q\right),
\end{equation}
for any $q \in(3,6)$, from \eqref{214} and \eqref{213}, we have
$$
t u_t \in C\left(\left[0, T_*\right] ; W^{1,4}\right),
$$
which implies that $u_t \in C\left(\left(0, T_*\right] \times \mathbb{T}^3\right)$.
Finally, we consider the continuity of div $\mathbb{S}$. Denote $\mathbb{N}= \varphi^2 L u- \nabla \varphi^2 \cdot Q(u)$. Based on \eqref{53} and \eqref{213}, we have
$$
t \mathbb{N} \in L^{\infty}\left(0, T_* ; H^2\right).
$$

From $\mathbb{N}_t \in L^2\left(0, T_* ; L^2\right)$ and \eqref{215}, we obtain $t \mathbb{N} \in C\left(\left[0, T_*\right] ; W^{1,4}\right)$, which implies that
$$\mathbb{N} \in C\left(\left(0, T_*\right] \times \mathbb{T}^3\right).$$

Since $\rho \in C\left(\left[0, T_*\right] \times \mathbb{T}^3\right)$ and $\operatorname{div} \mathbb{S}=\rho \mathbb{N}$, then we obtain the desired conclusion.
\end{proof}

\bigskip

\noindent {\bf Acknowledgments}\\
The authors sincerely appreciates Professor  Yachun Li for her helpful
suggestions and discussions on the problem solved in this paper. The research of this work was supported in part by the National Natural Science Foundation of China under grants 11831011, 12161141004, 12371221, and 11571232. This work was also partially supported by the Fundamental Research Funds for the Central Universities and Shanghai Frontiers Science Center of Modern Analysis, and by China Postdoctoral Science Foundation under grant 2021M692089.

\bigskip 
 
\noindent{\bf Data Availability Statements}\\ 
Data sharing not applicable to this article as no datasets were generated or analysed during the current study.

\bigskip

\noindent{\bf Conflict of interests}\\
The authors declare that they have no competing interests.

\bigskip

\noindent{\bf Authors' contributions}\\
The authors have made the same contribution. All authors read and approved the final manuscript.

\bigskip

\bibliographystyle{plain}
\bibliography{references}
\end{document}